\documentclass{article}
\usepackage{floatrow}
\newfloatcommand{capbtabbox}{table}[][\FBwidth]
\usepackage[pdftex,
bookmarksnumbered,
bookmarksopen,
colorlinks,
citecolor=blue,
linkcolor=blue,]{hyperref}
\usepackage{amsmath}
\usepackage{amssymb}
\usepackage{mathrsfs}
\usepackage{amsfonts}
\usepackage{amsthm}
\usepackage{booktabs}
\usepackage{listings}
\usepackage{boxedminipage}
\usepackage{algorithmic}
\usepackage{stmaryrd}
\usepackage{cite}
\usepackage{relsize}
\usepackage{pdfpages}

\usepackage[top=1.3in, bottom=1.5in, left=1.5in, right=1.5in]{geometry}

\newtheorem{remark}{Remark}[section]
\newtheorem{example}{Example}[section]

\newtheorem{assumption}{Assumption}[section]

\newtheorem{proposition}{Proposition}[section]
\newtheorem{theorem}{Theorem}[section]
\newtheorem{lemma}{Lemma}[section]

\newenvironment{mytabular}{\bgroup\tiny\tabular}{\endtabular\egroup}

\title{Efficiently Maximizing a Homogeneous Polynomial over Unit Sphere without Convex Relaxation 
}
\author{Yuning Yang\thanks{College of Mathematics and Information Science, Guangxi University, Nanning, 530004, China  (yuning.yang1207@gmail.com).} \and
	Guoyin Li\thanks{Department of Applied Mathematics, University of New South Wales, Sydney 2052, Australia
		(g.li@unsw.edu.au).}                                                              
}
 
\begin{document} 
\maketitle

\begin{abstract}
	This work studies the problem of maximizing a higher degree real homogeneous multivariate polynomial over the unit sphere. This problem is equivalent to finding the leading  eigenvalue  of the associated symmetric tensor of higher order, which is nonconvex and NP-hard. Recent advances show that semidefinite relaxation is quite effective to find a global solution. However, the solution methods involve full/partial eigenvalue decomposition during the iterates, which heavily limits its efficiency and scalability. On the other hand, for odd degree (odd order) cases, the order has to be increased to even, which potentially reduces the efficiency.
	
To find the global solutions, instead of convexifying the problem, we equivalently reformulate the problem as a nonconvex matrix program based on an equivalence property between symmetric rank-1 tensors and matrices of any order, which is a generalization of the existing results.  
	The program is directly solved by a vanilla alternating direction method, which only involves the computation of leading eigenvalue/singular value of certain matrices, benefiting from the special structure of the program.
	Although being nonconvex, under certain hypotheses,  it is proved that the algorithm    converges  to a   leading eigenvalue of the associated tensor.   Numerical experiments on different classes of tensors demonstrate that the proposed approach has a significant improvement in   efficiency and scalability, while it can keep the effectiveness of semidefinite relaxation as much as possible. For instance, the proposed method   finds the leading  eigenpair of a third-order $500$ dimensional Hilbert tensor in a personal computer within $100$ seconds.

\hspace{2mm}\vspace{3mm}

\end{abstract}

\section{Introduction}

Polynomial optimization receives much attention in recent years due to its ability of modeling problems arising from signal processing, engineering, material science and so on. Owing to its nonconvexity and NP-hardness in general, designing effective and efficient algorithms is a challenge. This paper is focused on solution methods for finding global optimizers to a class of   polynomial optimization, that is to maximize a real homogeneous multivariate polynomial of degree higher than three over a unit sphere. This class of problems is a hot topic and is of  importance,   because of its large connectivity with numerous real-world applications, and because   its special structure allows one to study it by exploiting its related tensor form. As a result, researchers from communities of optimization and   linear/nonlinear algebra have devoted their efforts to study it over the past decades.

 Given a coordinate system, a real $d$-th order $n$-dimensional tensor (hypermatrix)  $\mathcal A$ is a multi-way array consisting of $n^d$ entries  $\mathcal A_{i_1 i_2\cdots i_d}$ where every $i_j$ varies from $1$ to $n$.  A $2$-nd order tensor is a matrix, whereas a $1$-st order one is a vector.   $\mathcal A$ is symmetric if each entry is invariant under any permutation of its indices. As a    homogeneous quadratic form is uniquely corresponding to a matrix, a degree $d$ homogeneous polynomial uniquely determines a $d$-th order symmetric tensor as well; see, e.g., \cite{he2010approximation}. 
 From the tensor point of view, the maximal value of a real homogeneous multivariate polynomial over a unit sphere is equivalent to   the leading Z-eigenvalue or $\ell^2$-eigenvalue of the associated symmetric tensor\footnote{In what follows, without any confusion, we omit the prefix and simply call it the eigenvalue of a tensor.}, which was defined independently by Qi \cite{qi2005eigenvalues} and Lim \cite{lim2005singular}. When $d=2$, they reduce exactly to the matrix eigenvalues. Such problem finds applications  in independent component analysis \cite{comon1994independent},   quantum entanglement \cite{hu2015computing}, maximum-clique problems \cite{bomze2005quartic}, Bose-Einstein condensates \cite{hu2016note},     tensor decompositions \cite{wang2007successive} and tensor completion \cite{yang2015rank}, just to name a few; it also connects closely to tensor rank-1 approximation \cite{kolda2010tensor} and the computation of tensor spectral norm \cite{lim2014blind}.
Despite being commonly seen and applied, unfortunately, when $d\geq 3$, solving such a problem is   NP-hard in general \cite{he2010approximation,hillar2013most}. 

Several efforts have been made    to tackle this problem. The   power method and its generalizations has been generalized to higher-order tensors \cite{kofidis2002on,de2000on,kolda2011shifted,kolda2014adaptive}.   Methods  based on first- and second- order information have been developed in \cite{yu2016adaptive,jaffe2018newton,chang2016computing}. An alternating direction method of multipliers has been studied in \cite{jiang2014alternating}. \cite{che2017neural} proposed to solve the problem via neural networks. By exploring the link between a   polynomial and the related multilinear form, a maximum block improvement method was proposed in \cite{chen2012maximum}. Theoretically, the above methods are guaranteed to find a stationary point, namely, an eigenvalue of the associated tensor, but  not sure the leading one. Another category of methods aims at computing an approximation solution with a theoretically guaranteed lower bound \cite{ling2009biquadratic,luo2010semidefinite,he2010approximation,he2012probability,so2010deterministic,nie2012sum,zhou2012nonnegative,zhang2012cubic}, just to name a few. On the other side, methods for finding all eigenvalues have been proposed \cite{cui2014all,chen2016computing} when the size of the problem is not large. 

As a special instance of the general polynomial optimization problems, maximizing a homogeneous polynomial over the unit sphere can of course be solved by using the sum of squares (SOS) relaxation, based on  a hierarchy semidefinite programming (SDP) relaxation with increasing size \cite{l1-2001,parrilo2003semidefinite}. Although global solutions can be achieved in theory, the size of the SDPs to be solved grows exponentially fast in the size of the problem, resulting in computational inefficiency and inscalability. Tailored to the problem under consideration,  a method based on sequential SDP has been studied in \cite{hu2013finding} when $d$ is even. Since it is also an SOS relaxation method, the limitations of \cite{l1-2001,parrilo2003semidefinite} are inherited. In \cite{qi2009z},   Z-eigenvalue methods were developed to find   global solutions when the dimension is not very large. Different from the existing approaches, Jiang et al. \cite{jiang2015tensor} proposed to solve the problem   via a single SDP relaxation. Specifically, by revealing an equivalence link between symmetric rank-1 tensors and matrices, they proposed an SDP relaxation model, which was then solved by the alternating method of multipliers (ADMM). Although being a  relaxation problem, empirically it was almost always observed that the relaxation is tight, namely, solving the relaxation problem often yields global solutions to the problem under consideration. Based on the SOS representation, an SDP relaxation   was also proposed by Nie and Wang \cite{nie2014semidefinite}, which was solved by the   Matlab package SDPNAL \cite{zhao2010newton}. Although appearing in totally different formulations,   it was pointed out in \cite{hu2016note} that the SDPs in  \cite{jiang2015tensor} and \cite{nie2014semidefinite} are essentially the same. Benefiting from only solving a single SDP of fixed size, both \cite{jiang2015tensor} and \cite{nie2014semidefinite} are capable of dealing with problems of larger size than methods in \cite{l1-2001,parrilo2003semidefinite,hu2013finding}.

\paragraph{Motivation} Our study is motivated by two limitations of \cite{jiang2015tensor,nie2014semidefinite}. Firstly, although approaches in \cite{jiang2015tensor,nie2014semidefinite} are effective to find the global solutions, it is known that SDP is relatively more suitable for small and moderate size problems, while in the current setting, for instance,    a $4$-th order $n$-dimensional tensor   results in an SDP of size $O(n^2\times n^2)$, which becomes unsolvable when $n$ increases, say, $n\geq 100$. On the other hand,   the algorithm of \cite{jiang2015tensor} involves eigenvalue decomposition (EVD) at each iterate, which is of theoretically computational complexity $O(n^6)$ for the aforementioned size tensors. Similar phenomena happen  to the solver used in \cite{nie2014semidefinite}. As a result, these observations   affect the efficiency and scalability of the approaches of \cite{jiang2015tensor,nie2014semidefinite}. Secondly, both approaches in \cite{jiang2015tensor,nie2014semidefinite} are naturally designed for tensors of order even, while to deal with odd order cases, the order has to be augmented to even such that SDP  can be applied. This may   reduce   efficiency.

\paragraph{Contribution of this work} The goal is to improve the efficiency and scalability of the above two approaches for tensors of any order, while to keep the effectiveness as much as possible. To achieve this, we first generalize the equivalence property between rank-1 tensors and matrices developed in \cite{jiang2015tensor} to tensors of any order, which serves as a cornerstone of our approach. Based on this property, the problem of interest is   equivalently formulated as a matrix program with a matrix rank-1 constraint for tensors of any order. Such an  optimization model has the property that every stationary point, if existing, yields a global solution to the original problem.  In view of it, instead of solving its convex relaxation, we directly solve this nonconvex matrix program by using a vanilla ADMM. In contrast to performing an EVD, the main computation of each iterate involves   finding the leading (largest) eigenvalue /singular value  of a certain matrix,  whose computational complexity in theory is of $O(n^d)$ only, i,e., it is linear to the size of the data tensor.   Under a hypothesis on the dual variable, it is shown that the algorithm   converges to a \emph{global optimizer} of this nonconvex program, namely, a leading eigenvalue of the associated tensor. Moreover, to some extent, the algorithm can itself identify whether the result is globally optimal. It is then shown that the  hypothesis on the dual variable is closely related to the tightness of the SDP relaxation, which is not an isolated phnomenon.

Numerical experiments have demonstrated that the proposed approach has a large improvement concerning the efficiency and scalability, and can keep the effectiveness of convex relaxation for most tensors, especially for structured tensors. Our Matlab code is available online for public use.

The remainder is organized as follows. Tensor operations are introduced in Sect. \ref{sec:multilinear-algebra}. The nonconvex matrix program to be studied is formulated in Sect. \ref{sec:prob}, while its properties along with solution methods are investigated in Sect. \ref{sec:alg}. Numerical results will be presented in Sect. \ref{sec:numer}. Sect. \ref{sec:conc} draws some conclusions and remarks.

 \subsection{Notations and Tensor Operations}\label{sec:multilinear-algebra}
Vectors are written as boldface lowercase letters $(\mathbf x,\mathbf y,\ldots)$, matrices
correspond to italic capitals $(A,B,\ldots)$, and tensors are
written as calligraphic capitals $(\mathcal{A}, \mathcal{B},
\cdots)$. $\mathbb R$   denotes  the real   field.  $\mathbb R^{m\times n}$ denotes real matrices of dimension $m\times n$ and $\mathbb S^{n\times n}$ denotes symmetric matrices of dimension $n\times n$.
 
A tensor is a multi-way array. A real $d$-th order $(n_1,\ldots,n_d)$-dimensional tensor $\mathcal A$ is defined as $\mathcal A = (A_{i_1\cdots i_d})_{1\leq i_1\leq n_1,\ldots,1\leq i_d\leq n_d}$, whose space is denoted as $\mathbb R^{n_1\times\cdots\times n_d}$. When $n_1=\cdots =n_d$, we write $\mathbb R^{n^d}$ for short. 
 For two tensors $\mathcal A,\mathcal B$ of the same size, their inner product is given by
$\langle \mathcal A,\mathcal B\rangle = \sum^{n_1}_{i_1=1}\cdots\sum^{n_d}_{i_d=1}\mathcal A_{i_1\cdots i_d}\mathcal B_{i_1\cdots i_d}.$
The Frobenius (or Hilbert–Schmidt)  norm of $\mathcal A$ is defined by $\|\mathcal A\|_F = \langle\mathcal A,\mathcal A\rangle^{1/2}.$  ${\rm tr}(\cdot)$ denotes the trace of a square matrix. It holds that ${\rm tr}(A)=\langle A,I \rangle $, where $I$ denotes the identity matrix of the same size as $A$.

\paragraph{Permutation}
Let $\pi(1,\ldots, d)$ be the sets of all permutations of $\{1,\ldots, d \}$. For any $\pi\in \pi(1,\ldots, d)$, define $\boldsymbol{ Per }_{[\pi]}(\mathcal A)$ as a permutation of $\mathcal A$ with respect to $\pi$. For example, for $A\in\mathbb R^{n^3}$, $\boldsymbol{ Per }_{[2;1;3]}(\mathcal A)$ is given by switching the first and the second modes of $\mathcal A$. Reducing to the matrix case, $\boldsymbol{ Per }_{[2;1]}(\cdot)$ is exactly the matrix transpose. It is identical to the Matlab function   \textsf{permute}.

\paragraph{Symmetric    tensors and symmetrization} For   $\mathcal A \in\mathbb R^{n^d}$,  if $\boldsymbol{ Per }_{[\pi]}(\mathcal A) = \mathcal A$ for any $\pi\in \pi(1,\ldots, d)$, 
 then $\mathcal A$ is called symmetric. 
 $\mathbb S^{n^d}$   denotes the subspace of $d$-th order $n$-dimensional real symmetric tensors. 

For any $\mathcal A\in \mathbb R^{n^d}$, define the symmetrization of $\mathcal A$  as $$\small\boldsymbol{Sym}(\mathcal A):= \frac{\sum_{\pi \in \pi(1,\ldots,d)  } \boldsymbol{Per}_{[\pi]}   (\mathcal A)    }  {d!} \in\mathbb S^{n^d},$$ namely, it is given by the average of the sum of all the permutations of $\mathcal A$. It then can be readily seen that
\begin{proposition}\label{prop:sym_invariance}
	Let $\mathcal A\in\mathbb S^{n^d}$. Then for any $\mathcal B\in\mathbb R^{n^d}$, there holds $\langle \mathcal A,\mathcal B\rangle = \langle\mathcal A,\boldsymbol{Sym}(\mathcal B)\rangle$.
\end{proposition}

\paragraph{Matricization and tensorization}\label{Sec:balanced_unfoldings}  Given $\mathcal A\in\mathbb R^{n^d}$, we define     $\boldsymbol{ Mat }(\mathcal A) $ as a matrix following the Matlab function \textsf{reshape}, i.e., $\boldsymbol{ Mat }(\mathcal A):= \textsf{reshape}(\mathcal A, n^{ \lfloor d/2\rfloor  }, n^{ \lceil d/2 \rceil  }) \in \mathbb R^{ n^{ \lfloor d/2\rfloor }\times n^{\lceil d/2\rceil}      }$. When $d$ is even, $\boldsymbol{ Mat }(\mathcal A) $ is a symmetric matrix. Conversely, for any $A \in \mathbb R^{ n^{ \lfloor d/2\rfloor }\times n^{\lceil d/2\rceil}      }$, define the tensorization of $A$ as $\boldsymbol{ Ten }(A) := \textsf{reshape}(A,\overbrace{n,\ldots,n}^{d}) \in\mathbb R^{n^d}$. There hold $\mathcal A=\boldsymbol{ Ten }( \boldsymbol{ Mat }(\mathcal A)  )$ and $A=\boldsymbol{ Mat }(\boldsymbol{ Ten }(A))$.

  \paragraph{Outer  product and Kronecker product} Notations follow those of \cite{kolda2010tensor}.
  The outer  product $\mathbf a_1\circ   \cdots\circ \mathbf a_d$ of $d$ vectors $\mathbf a_i\in\mathbb R^{n_i},~ 1\leq i\leq d$ is a rank-1 tensor  whose entries are the product of the corresponding vector entries:
  $$(\mathbf a_1\circ   \cdots\circ\mathbf a_d)_{j_1\cdots j_d} = \mathbf a_{1,{j_1}} \cdots \mathbf a_{d,{j_d}},~\forall~ 1\leq j_i\leq n_i,~1\leq i\leq d.$$
When $d=2$, it reduces to the multiplication of a column vector and a row vector, namely,  $
  \mathbf a_1 \circ \mathbf a_2 =\mathbf a_1\mathbf a_2^{\top}.
  $
The Kronecker product $\mathbf a_1\otimes\cdots  \otimes\mathbf a_d\in\mathbb R^{ \prod^d_{i=1}n_i }$  is  a vector given by the vectorization of   $\mathbf a_1\circ\cdots\circ \mathbf a_d $. Throughout this paper, we use  $\mathbf a^{\circ^m}$ and $\mathbf a^{\otimes^m}$ to represent $\overbrace{\mathbf a \circ\cdots\circ\mathbf a}^{m}$ and $\overbrace{\mathbf a\otimes\cdots \otimes\mathbf a}^{m}$, respectively.

\paragraph{Tensor CP-rank} 
The CP-rank of a tensor $\mathcal A$, denoted by ${\rm rank_{CP}}(X)$, is defined as the
smallest number of rank-$1$ tensors that generate $\mathcal A$ as their sum \cite{kolda2010tensor}.  In particular, we write ${\rm rank_{CP}}(\mathcal A)=1$ if $\mathcal A$ is a rank-1 tensor.

\section{Problem Formulation}\label{sec:prob}
The optimization model we are interested is 
\begin{equation}
\label{prob:original2}
\max ~f(\mathbf x) = \sum^n_{i_1,\ldots,i_d=1}\nolimits\mathcal A_{i_1\cdots i_d}\mathbf x_{i_1}\cdots \mathbf x_{i_d} =   \langle \mathcal A ,   \mathbf x^{\circ^d} \rangle~{\rm s.t.} \|\mathbf x\|=1,  \mathbf x\in  \mathbb R^n.
\end{equation}
Denoting $\mathcal A\mathbf x^{\circ^{d-1}} \in\mathbb R^{n}$ with $(\mathcal A\mathbf x^{\circ^{d-1}} )_i = \sum^n_{i_2,\ldots,i_d=1}\mathcal A_{ii_2\cdots i_d}\mathbf x_{i_2}\cdots\mathbf x_{i_d}$, the eigenvalue problem of $\mathcal A$ is defined as follows:
\[
   \mathcal A\mathbf x^{\circ^{d-1}}   = \sigma\mathbf x,~\|\mathbf x\|=1,
\]
where $\sigma\in\mathbb R$. Here $(\mathbf x,\sigma)$ is called an eigenpair of $\mathcal A$. Then the maximum of  \eqref{prob:original2} is the leading/largest eigenvalue of $\mathcal A$.
We prefer to write \eqref{prob:original2} as a tensor optimization problem. 
Denote $\mathcal X:=\mathbf x^{\circ^d}\in\mathbb S^{n^d}$. When $d$ is even, $\boldsymbol{ Mat }(\mathcal X) =\mathbf x^{\otimes^{d/2}}\mathbf x^{\otimes^{d/2}\top}$, and it holds that $\|\mathbf x\|=1\Leftrightarrow \| \mathbf x^{\otimes^{d/2}}\|=1\Leftrightarrow {\rm tr}( \boldsymbol{ Mat }(\mathcal X)  ) =1$; when $d$ is odd, we have $\|\mathbf x\|=1\Leftrightarrow \|\mathcal X\|_F=1$. Now we can equivalently rewrite \eqref{prob:original2} into the following form with a tensor variable:
\begin{equation}
\label{prob:original3}
\max~ \langle \mathcal A,\mathcal X\rangle~{\rm s.t.} ~{\rm rank_{CP}} (\mathcal X)=1, \mathcal X\in \mathbb S^{n^d}, \left\{  
\begin{array}{lr}
  {\rm tr}( \boldsymbol{ Mat }(\mathcal X)  ) =1,  & d~{\rm is~even},\\  
 \|\mathcal X\|_F=1, & d~{\rm is~odd}.\\
\end{array}  
\right.  
\end{equation}
which is a tensor optimization problem.  Note that we have distinguished the even and odd cases, because when $d$ is odd and if $\mathcal X$ is a feasible solution, then so is $-\mathcal X$, which is not true when $d$ is even. 

When $d$ is even, Jiang et al. \cite{jiang2015tensor} showed that for $\mathcal X\in\mathbb S^{n^d}$, if $\boldsymbol{ Mat }(\mathcal X)$ is a rank-1 matrix, then $\mathcal X$ itself is also a rank-1 tensor. Using this equivalence property, the constraint ${\rm rank_{CP}}(\mathcal X)=1$ can be equivalently replaced by ${\rm rank}(\boldsymbol{ Mat }(\mathcal X))=1$ in \eqref{prob:original3}, and then \eqref{prob:original3}   turns into a matrix program with matrix rank-1 constraint. Based on this property, an SDP relaxation was proposed in \cite{jiang2015tensor} provided $d$ being even (note that $\boldsymbol{ Mat }(\mathcal X) \in\mathbb S^{n^{d/2}\times n^{d/2}}$ ):
 \begin{equation}
\label{prob:relax_convex}
(R)~~\max ~\langle  \mathcal A,\mathcal X\rangle~{\rm s.t.}~        \boldsymbol{Mat}(\mathcal X)\succeq 0, ~\mathcal X\in\mathbb S^{n^d},~{\rm tr}(\boldsymbol{Mat}(\mathcal X))  =1,
\end{equation}
Although the SDP    is   effective to find the global solutions of \eqref{prob:original3}  in practice, solving it typically relies on computing   full/partial EVD of  size $n^{d/2}\times n^{d/2}$ at each iterate. On the other hand,   dealing with odd $d$ requires to increase $d$ \cite{jiang2015tensor}. Similar situations happen   to the models in \cite{nie2014semidefinite}.

In view of the above limitations, we consider a nonconvex reformulation of \eqref{prob:original3}. To achieve it, in the following we first   give a generalization of   the equivalence property of \cite{jiang2015tensor} to any order $d$. It can be seen as a corollary of \cite[Theorem 3.3 and Corollary 3.5]{yang2016rank}. For completeness, we  present a concise proof in the supplemental materials.
\begin{theorem}\label{prop:equivalance} For any integer $d\geq 2$, there holds
\begin{equation*} 
\label{eq:equivalence}
		      \setlength\abovedisplayskip{2pt}
\setlength\abovedisplayshortskip{2pt}
\setlength\belowdisplayskip{2pt}
\setlength\belowdisplayshortskip{2pt}
\{ \mathcal X~|~ {\rm rank_{CP}}  (\mathcal X)=1,\mathcal X\in\mathbb S^{n^d}    \} = \{ \mathcal X~|~ {\rm rank}( \boldsymbol{ Mat }(\mathcal X)   )=1,  \mathcal X\in \mathbb S^{n^d}\}.
\end{equation*}
\end{theorem}

 Based on Theorem \ref{prop:equivalance}, it is readily seen that problem   \eqref{prob:original3} is equivalent to the following problem  with a matrix rank-1 constraint, for any order $d$:
 \begin{equation}
 \label{prob:original4}
 (O)~~\max ~\langle  \mathcal A,\mathcal X\rangle~{\rm s.t.} ~  {\rm rank}(\boldsymbol{Mat}(\mathcal X))=1, \mathcal X\in \mathbb S^{n^d}, \left\{  \begin{array}{lr}
 {\rm tr}( \boldsymbol{ Mat }(\mathcal X)  ) =1,  & d~{\rm is~even},\\  
 \|\mathcal X\|_F=1, & d~{\rm is~odd}.\\
 \end{array}  
 \right.  
 \end{equation}
Once a global solution $\mathcal X^*$ of \eqref{prob:original4} is obtained, Theorem \ref{prop:equivalance} shows that $\mathcal X^*$ is a rank-1 tensor. By writing     $\mathcal X^*=\mathbf x^{\circ^d}$ with $\|\mathbf x\|=1$, $\mathbf x$ is thus a global solution to the original problem \eqref{prob:original2}.
 \eqref{prob:original4}   is the main model of this work. Its properties and solution methods will be studied in the next section.

\section{Optimality Conditions, Nonconvex ADMM, and Convergence}\label{sec:alg}
We remark that all the models, algorithms and theoretical results presented in this section are applicable for all $d$, while we mainly present the analysis when $d$ is even, because it is quite similar when $d$ is odd\footnote{When $d$ is odd, its convex relaxation can be as follows, although we do not   solve it. Here $\|\cdot\|_*$ stands for the nuclear norm of a matrix, i.e., the sum of singular values.
	\begin{equation*}
	\label{prob:relax1}
	\setlength\abovedisplayskip{2pt}
	\setlength\abovedisplayshortskip{2pt}
	\setlength\belowdisplayskip{2pt}
	\setlength\belowdisplayshortskip{2pt}
	\min~\langle-\mathcal A,\mathcal X\rangle~{\rm s.t.}~ \|\boldsymbol{ Mat }(\mathcal X)\|_*\leq 1, \mathcal X\in\mathbb S^{n^d}.
	\end{equation*}}.
\subsection{Global optimality Conditions and ADMM}

We first introduce an auxiliary variable $\mathcal Y$ and equivalently rewrite \eqref{prob:original4} as
\begin{equation}
\tag{\ref{prob:original4}$^\prime$}  \label{prob:relax3}
		      \setlength\abovedisplayskip{2pt}
\setlength\abovedisplayshortskip{2pt}
\setlength\belowdisplayskip{2pt}
\setlength\belowdisplayshortskip{2pt}
\min ~\langle  -\mathcal A,\mathcal Y\rangle~{\rm s.t.}~ \mathcal X=\mathcal Y, \mathcal X\in C,~\mathcal Y\in\mathbb S^{n^d},
\end{equation}
where for convenience we use ``$\min$'' to replace ``$\max$'', and
$$\small C:= 	\left\{  
\begin{array}{lr}
\{\mathcal X\in\mathbb R^{n^d}\mid {\rm rank}(\boldsymbol{Mat}(\mathcal X))=1 ,\boldsymbol{ Mat }(\mathcal X)\in \mathbb S^{ n^{   d/2   }\times n^{   d/2   }  },  {\rm tr}( \boldsymbol{ Mat }(\mathcal X)  ) =1  \},  & d~{\rm   even},\\  
\{\mathcal X\in\mathbb R^{n^d}\mid   {\rm rank}(\boldsymbol{Mat}(\mathcal X))=1 , \|\mathcal X\|_F=1  \},& d~{\rm  odd}.\\
\end{array}  
\right.  
$$
The purpose  of \eqref{prob:relax3} is to seperate   $C\cap S^{n^d}$ into $C$ and $S^{n^d}$ that can   easier   deal with.
\begin{remark}\label{rmk:1}
The constraint $C$ is equivalent to that when $d$ is even, $\boldsymbol{Mat}(\mathcal X) = \mathbf z\mathbf z^\top$ with $\|\mathbf z\|=1,\mathbf z\in\mathbb R^{n^{d/2}}$; when $d$ is odd,  $\boldsymbol{Mat}(\mathcal X) = \mathbf z_1\mathbf z_2^\top$, with $\|\mathbf z_1\|=\|\mathbf z_2\|=1,\mathbf z_1\in\mathbb R^{n^{ \lfloor d/2\rfloor   }  }, \mathbf z_2\in\mathbb R^{ n^{ \lceil d/2 \rceil   }   }$.
\end{remark}

The Lagrangian function for \eqref{prob:relax3} is given as
$L(\mathcal X,\mathcal Y,\Lambda) := \langle -\mathcal A,\mathcal Y\rangle -\langle \Lambda,\mathcal X-\mathcal Y\rangle $, with   $\Lambda \in \mathbb R^{n^d}$ being the dual variable.   The optimality condition for \eqref{prob:relax3} reads as follows:
\begin{equation}  \label{eq:kkt_nonconvex_reformulation_notau}\small
	\left\{  
	\begin{array}{lr}
\mathcal X^* \in \arg\min_{\mathcal X\in C} L(\mathcal X,\mathcal Y^*,\Lambda^*)=\arg\min_{\mathcal X\in C} \langle -\Lambda^*,\mathcal X\rangle	\Leftrightarrow \langle -\Lambda^*, \mathcal X-\mathcal X^*\rangle \geq 0   ,  &\forall \mathcal X\in C,\\  
 \mathcal Y^*\in\arg\min_{\mathcal Y\in\mathbb S^{n^d}}	L(\mathcal X^*,\mathcal Y,\Lambda^*)\Leftrightarrow			 \langle -\mathcal A+\Lambda^*, \mathcal Y-\mathcal Y^*\rangle\geq 0, &\forall \mathcal Y\in  \mathbb S^{n^d} ,  \\  
 ~~~~~~~~~~~~~~~~~~~~~~~~~~~~~~~~~~~~~~\,~~~\Leftrightarrow \boldsymbol{ Sym}(\Lambda^*)=\mathcal A,&\\
	\mathcal X^*=\mathcal Y^*,&\\  
	\end{array}  
	\right.  
	\end{equation}  
Due to  Theorem \ref{th:global_solution}, the $\mathcal Y$-subproblem amounts to $\boldsymbol{ Sym }(\Lambda^*)=\mathcal A$.   Note that although $C$ is nonconvex, the $\mathcal X$-subproblem can   be written as a variational inequality as well because the Lagrangian is linear with respect to $\mathcal X$.
 	By eliminating $\mathcal Y^*$, \eqref{eq:kkt_nonconvex_reformulation_notau}  can be   simplified as follows:
\begin{equation}
\label{eq:kkt_nonconvex_reformulation_notau_variant}
\exists ~\{\mathcal X^*,\Lambda^* \} \in  C\cap \mathbb S^{n^d}\times \mathbb R^{n^d},~{\rm s.t. }~ \mathcal X^*\in\arg\min_{\mathcal X\in C}\nolimits \langle -\Lambda^*,\mathcal X\rangle,~\boldsymbol{ Sym}(\Lambda^*) = \mathcal A. 
\end{equation}
Since $\mathcal X^*\in C\cap \mathbb S^{n^d}$, Theorem \ref{prop:equivalance} shows that $\mathcal X^*$ can be written as $\mathcal X^*=\mathbf x^{\circ^d}$ with   $\|\mathbf x\|=1,\mathbf x\in\mathbb R^n$.

\begin{remark}\label{rmk:2}  \eqref{eq:kkt_nonconvex_reformulation_notau_variant} means the existence of $\Lambda^*$ such that $\boldsymbol{ Sym}(\Lambda^*) = \mathcal A$, with $\mathbf x^{\otimes^{d/2}}$ being a leading eigenvector of $\boldsymbol{ Mat }(\Lambda^*)$ ($d$ is even). The most simple example is the orthogonally decomposable tensor, in which one has $\Lambda^*=\mathcal A$, and  if $\mathbf x$ is a leading eigenvector of $\mathcal A$, then $\mathbf x^{\otimes^{d/2}}$ is also a leading eigenvector of $\boldsymbol{ Sym}(\Lambda^*) $. On the other hand, when $d=2$, i.e., the matrix cases, \eqref{eq:kkt_nonconvex_reformulation_notau} naturally holds with $\Lambda^*=\mathcal A$. However, in general, $\Lambda^*\neq \mathcal A$.
\end{remark}	
	\begin{example}
 Consider the example that $\mathcal A\in\mathbb S^{2^4}$ with all the entries of $\mathcal A$ being one, except   $\mathcal A_{1111}=2$. We have $\Lambda^*$ with
$$\footnotesize
\boldsymbol{ Mat }(\Lambda^*)=	\left[
\begin{matrix}
2.0000 &    1.0000  &  1.0000  &  0.7349 \\
1.0000  &  1.1326 &   1.1326 &   1.0000 \\
1.0000  &  1.1326 &   1.1326 &   1.0000\\
0.7349  &  1.0000 &   1.0000 &   1.0000
\end{matrix}
\right]  ,
$$
and $\mathcal X^*=\mathbf x^{\circ^4}$ where $\mathbf x = [0.7557~ 0.6550]^\top$. $\{\mathcal X^*,\Lambda^* \}$ satisfies \eqref{eq:kkt_nonconvex_reformulation_notau_variant} while $\Lambda^*\neq\mathcal A$.
	\end{example}

It should be noted that due to the nonconvexity of $C$, in general, it is hard to determine whether solutions exist to  \eqref{eq:kkt_nonconvex_reformulation_notau}, namely, it is not sure whether $\Lambda^*$ exists. This issue will be further studied in  Section \ref{sec:uniqueness}.
Here, the following proposition shows the connection between the existence and the tightness of the convex relaxations.
	\begin{proposition}\label{prop:admm:1}
		When $d$ is even,   \eqref{eq:kkt_nonconvex_reformulation_notau} exists solutions iff the SDP relaxation \eqref{prob:relax_convex} is tight.
	\end{proposition} 

Its proof is left behind  Proposition \ref{prop:cond:1} in   Section \ref{sec:uniqueness}. 
As a result of this proposition and that the convex relaxation is often tight in practice \cite{jiang2015tensor}, the system \eqref{eq:kkt_nonconvex_reformulation_notau} is reasonable. Thus  in the sequel, our study is based on  the existence of a solution $\{\mathcal X^*,\mathcal Y^*,\Lambda^* \}$  to \eqref{eq:kkt_nonconvex_reformulation_notau}. 

The following shows that despite   the nonconvexity of $C$,   $\mathcal X^*$ is in fact a global solution to \eqref{prob:original4}.
\begin{theorem}
	\label{th:global_solution}
Let  $\{ \mathcal X^*,\mathcal Y^*,\Lambda^* \}$ satisfy the optimality condition \eqref{eq:kkt_nonconvex_reformulation_notau}. Then 
\begin{enumerate}
\item $\boldsymbol{Sym}(\Lambda^*) = \mathcal A$;
\item $\mathcal X^*$ is an optimal solution to \eqref{prob:original4}.
\end{enumerate}
\end{theorem}
\begin{proof}
For any $\mathcal Z\in\mathbb S^{n^d}$, using Proposition \ref{prop:sym_invariance} we have $\langle \boldsymbol{Sym}(-\mathcal A + \Lambda^*), \mathcal Z\rangle = \langle -\mathcal A+\Lambda^*,\mathcal Z\rangle$; while   the second inequality in \eqref{eq:kkt_nonconvex_reformulation_notau} means that $ \langle -\mathcal A+\Lambda^*,\mathcal Z\rangle=0$ for any $\mathcal Z\in\mathbb S^{n^d}$. This holds iff $\boldsymbol{Sym}(\Lambda^*)=\mathcal A$.

As $\mathcal X^*\in C$ and $\mathcal X^*=\mathcal Y^*\in\mathbb S^{n^d}$, it follows from Theorem \ref{prop:equivalance} that $\mathcal X^*$ is a   rank-1 tensor. Assume that $\mathcal X^*$ is not optimal to \eqref{prob:original4}; then there exists another feasible solution $\mathcal Z $ of  \eqref{prob:original4} such that $\langle \mathcal A,\mathcal Z\rangle > \langle \mathcal A,\mathcal X^*\rangle$. Using again Proposition \ref{prop:sym_invariance}, this results in $\langle -\Lambda^*,\mathcal Z\rangle <\langle -\Lambda^*,\mathcal X^*\rangle$. As $\mathcal Z$ is feasible to \eqref{prob:original4}, $\mathcal Z\in C$, which contradicts the first inequality of \eqref{eq:kkt_nonconvex_reformulation_notau}. The  proof has been completed.
\end{proof}

We propose to solve \eqref{prob:relax3} via a nonconvex ADMM, which relies on the augmented Lagrangian function   defined as  
\begin{equation}\label{eq:augmented_lag}
		      \setlength\abovedisplayskip{2pt}
\setlength\abovedisplayshortskip{2pt}
\setlength\belowdisplayskip{2pt}
\setlength\belowdisplayshortskip{2pt}
L_\tau(\mathcal X,\mathcal Y,\Lambda) := \langle -\mathcal A,\mathcal Y\rangle -\langle \Lambda,\mathcal X-\mathcal Y\rangle + \frac{\tau}{2}\|\mathcal X-\mathcal Y\|_F^2,
\end{equation} 
where $\tau>0$. 
Due to the nonconvexity of $C$, the optimality condition derived from \eqref{eq:augmented_lag}, especially the $\mathcal X$-subproblem is, however, slightly different from \eqref{eq:kkt_nonconvex_reformulation_notau}, and is given as follows:
\begin{equation}  \label{eq:kkt_nonconvex_reformulation}
\left\{  
\begin{array}{lr}
\mathcal X^*\in\arg\min_{\mathcal X\in C} L_\tau(\mathcal X,\mathcal Y^*,\Lambda^*) \Leftrightarrow\langle -\Lambda^*\color{blue}{-\tau\mathcal Y^*}\color{black}, \mathcal X-\mathcal X^*\rangle \geq 0,  &\forall \mathcal X\in C,\\  
\mathcal Y^*\in\arg\min_{\mathcal Y\in\mathbb S^{n^d}}	L_\tau(\mathcal X^*,\mathcal Y,\Lambda^*)    \Leftrightarrow    \boldsymbol{ Sym}(\Lambda^*)=\mathcal A ,  \\  
\mathcal X^*=\mathcal Y^*.&\\  
\end{array}  
\right.  
\end{equation}  
 The $\mathcal X$-subproblem can still be written as a variational inequality   because the Lagrangian is linear with respect to $\mathcal X$ under the constraint $C$  (the quadratic term $\|\mathcal X\|_F^2$ of $L_\tau$ is the constant $1$ under $C$, either $d$ is even or odd). The term $-\tau\mathcal Y^*$ in the $\mathcal X$-subproblem causes a little trouble: Let $\{\mathcal X^*,\mathcal Y^*,\Lambda^*\}$ meet \eqref{eq:kkt_nonconvex_reformulation_notau}; then it must satisfy \eqref{eq:kkt_nonconvex_reformulation}, but the converse might not be true. Nevertheless, for any solution of \eqref{eq:kkt_nonconvex_reformulation}, when writing $\mathcal X^*=\mathbf x^{\circ^{d}}$ with $\mathbf x\in\mathbb R^n,~\|\mathbf x\|=1$ , $\mathbf x$  still acts as an eigenvector of $\mathcal A$:
\begin{proposition}\label{prop:kkt_withtau}
Let $\{\mathcal X^*,\mathcal Y^*,\Lambda^*\}$ meet \eqref{eq:kkt_nonconvex_reformulation}. Denote $\mathcal X^*=\mathbf x^{\circ^{d}}$ with $\|\mathbf x\|=1$. Then there holds $\mathcal A\mathbf x^{\circ^{d-1}} = \sigma\mathbf x$ for some $\sigma\in\mathbb R$.
\end{proposition}
\begin{proof}
	Denote $\mathcal A_\tau:= \mathcal A+\tau\mathcal X^*$ and $\Lambda_\tau := \Lambda^*+\tau\mathcal X^*$. Similar to the proof of Theorem \ref{th:global_solution}, we can show that $\mathcal X^*$ maximizes $\langle \mathcal A_\tau,\mathcal X\rangle$ over all $\mathcal X$ feasible to \eqref{prob:original4} (or \eqref{prob:original3}). This implies that   $\mathcal A_\tau \mathbf x^{\circ^{d-1}} = \sigma_\tau\mathbf x$ with $\sigma_\tau = \langle\mathcal A_\tau,\mathcal X^*\rangle$, which is essentially $\mathcal A\mathbf x^{\circ^{d-1}} = \sigma\mathbf x$ with $\sigma=\sigma_\tau - \tau$, as desired.
\end{proof}



Even though Proposition \ref{prop:kkt_withtau} tells us that   a   tuple  $\{\mathcal X^*,\mathcal Y^*,\Lambda^*\}$  of \eqref{eq:kkt_nonconvex_reformulation} might not satisfy \eqref{eq:kkt_nonconvex_reformulation_notau}, the following results illustrate that   if the leading eigenvalue of $\boldsymbol{Mat}(\Lambda^*)$ is simple, and $\sigma$ is properly chosen, then $\{\mathcal X^*,\mathcal Y^*,\Lambda^*\}$ is still a solution to \eqref{eq:kkt_nonconvex_reformulation_notau}. 
\begin{theorem}\label{prop:proof:1} Let $d$ be even.	Let $\{ \mathcal X^*,\mathcal Y^*,\Lambda^* \}$ satisfy the system \eqref{eq:kkt_nonconvex_reformulation}. Let $\{ \sigma_i \},i=1,\ldots,n^{d/2}$ denote the eigenvalues of    $-\boldsymbol{Mat}(\Lambda^*)$, arranged in a descending order; assume that the smallest eigenvalue $\sigma_{n^{d/2}} $ is a simple root. If $\tau$ satisfies
	\begin{equation}\label{eq:proof:10}
	0<\tau < \sigma_{n^{d/2}-1} -\sigma_{n^{d/2}}   =:\beta,
	\end{equation}
	then there holds
	\begin{equation}\label{eq:proof:15}
			      \setlength\abovedisplayskip{2pt}
	\setlength\abovedisplayshortskip{2pt}
	\setlength\belowdisplayskip{2pt}
	\setlength\belowdisplayshortskip{2pt}
	\langle-\Lambda^*,\mathcal X-\mathcal X^*\rangle
	\geq 0,~~\forall \mathcal X\in  C,
	\end{equation}
	i.e., $\{\mathcal X^*,\mathcal Y^*,\Lambda^*\}$ is still a solution to \eqref{eq:kkt_nonconvex_reformulation_notau}. 
\end{theorem}
\begin{proof}
When $d$ is even, from the form of \eqref{eq:kkt_nonconvex_reformulation}, without loss of generality we can assume that the matrix $\boldsymbol{ Mat }(\Lambda^*)$ is symmetric. Then the $\mathcal X$-subproblem of \eqref{eq:kkt_nonconvex_reformulation} is essentially finding the smallest eigenvalue of $-\boldsymbol{ Mat }(\Lambda^*+\tau\mathcal Y^*)$. Let $\sigma^\prime_{\min}$ be the smallest eigenvalue of $-\boldsymbol{Mat}(\Lambda^*+\tau\mathcal Y^*)$, namely, $-\boldsymbol{Mat}(\Lambda^*+\tau\mathcal X^*)$. Since $\mathcal X^* \in C$, we can write $\boldsymbol{Mat}(\mathcal X^*) =  \mathbf z^*\mathbf z^{*\top}$, with $\mathbf z^*\in\mathbb R^{n^{d/2}},~\|\mathbf x^*\|=1$. One observes that
	\begin{eqnarray}\label{eq:proof:16}
		\langle -\Lambda^*-\tau\mathcal X^*,\mathcal X^*\rangle = \sigma^\prime_{\min} &\Leftrightarrow& -( \boldsymbol{Mat}(\Lambda^*)+\tau \mathbf z^*\mathbf z^{*\top})\mathbf z^* = \sigma_{\min}^\prime\mathbf z^*\nonumber\\
		&\Leftrightarrow& -\boldsymbol{Mat}(\Lambda^*)\mathbf z^* = (\tau+\sigma_{\min}^\prime)\mathbf z^*,
	\end{eqnarray}
	telling us that $\mathbf z^*$ is also an eigenvector of $-\boldsymbol{Mat}(\Lambda^*)$. In what follows, we assert that    $\mathbf z^*$ corresponds to $\sigma_{n^{d/2}}$, namely, $\mathcal X^*$ is optimal to $\min_{\mathcal X\in C}\langle-\Lambda^*,\mathcal X\rangle $. Otherwise, as   $\sigma_{n^{d/2}} $     is simple, it follows from \eqref{eq:proof:16} that $\tau+\sigma^\prime_{\min}\neq \sigma_{n^{d/2}}$; then we must have
	\begin{eqnarray*}
		\tau+\sigma_{\min}^\prime &\geq& \sigma_{n^{d/2}-1} \\
		\Leftrightarrow ~~~\sigma_{\min}^\prime &\geq& \sigma_{n^{d/2}-1}  - \tau 
		 >  \sigma_{n^{d/2}-1} - (\sigma_{n^{d/2}-1} -\sigma_{n^{d/2} }) 
		 =  \sigma_{n^{d/2} }. 
	\end{eqnarray*}
	On the other hand, assume that $\mathbf y$ is the eigenvector corresponding to $\sigma_{ n^{d/2} }$ of $-\boldsymbol{Mat}(\Lambda^*)$. Then $\mathbf y^{\top}\mathbf z^*=0 \Leftrightarrow \langle \boldsymbol{Mat}(\mathcal X^*),\mathbf y\mathbf y^{\top}\rangle = 0$, and so
	\begin{equation*}
	\langle -\boldsymbol{Mat}(\Lambda^*+ \tau X^*),\mathbf y\mathbf y^{\top}\rangle  =\sigma_{ n^{d/2} } < \sigma^\prime_{\min} = \langle -\Lambda^*-\tau\mathcal X^*,\mathcal X^*\rangle,
	\end{equation*}
	which contradicts that $\mathcal X^*$ is optimal to $\min_{\mathcal X\in C}\langle -\Lambda^*-\tau  \mathcal X^* , \mathcal X \rangle  $. As a result, $\mathbf x^*$ corresponds to the smallest eigenvalue of $-\boldsymbol{Mat}(\Lambda^*)$, and hence \eqref{eq:proof:15} is true. 
\end{proof}

The nonconvex ADMM for solving \eqref{prob:relax3} is presented as follows. Note that when $d$ is even,   the initializer $\Lambda^0$ should satisfy that $\boldsymbol{ Mat }(\Lambda^0) \in\mathbb S^{n^{d/2}\times n^{d/2}  }$ when $d$ is even; and $\mathcal Y^0\in\mathbb S^{n^d}$. This together with the definition of the algorithm yields that $\boldsymbol{ Mat }(\Lambda^{k} + \tau\mathcal Y^{k})$ is always a symmetric matrix for all $k$, and so
  the $\mathcal X$-subproblem amounts to a symmetric matrix eigenvalue problem. Usually we set $\Lambda^0=\mathcal A$.

 \begin{boxedminipage}{0.97\textwidth}\small
Nonconvex ADMM for solving \eqref{prob:relax3}/\eqref{prob:original4}
	\begin{eqnarray} \label{alg:admm}
	\mathcal X^{k+1} &\in& \arg\min_{\mathcal X\in C   } L_\tau(\mathcal X,\mathcal Y^{k} ,\Lambda^{k}    )\Leftrightarrow   \langle -\Lambda^{k} - \tau\mathcal Y^{k},\mathcal X - \mathcal X^{k+1}\rangle\geq 0,~\forall \mathcal X\in C,\nonumber\\
	&\Leftrightarrow &\mathcal X^{k+1}\in \arg\max_{\mathcal X\in C}  \langle  \Lambda^{k} + \tau\mathcal Y^{k},  \mathcal X^{k+1}\rangle \nonumber\\
	& \Leftrightarrow& \boldsymbol{ Mat }(\mathcal X^{k+1})=  \left\{  \begin{array}{lr}
	\mathbf x\mathbf x^\top ,~\mathbf x~{\rm is~a~leading~eigenvector}&\\
	 ~~~~~~~~~~~~~~~~~~~~{\rm~of~} \boldsymbol{Mat}( \Lambda^{k} + \tau\mathcal Y^{k}) & d~{\rm  even},\\  
	\mathbf x_1\mathbf x_2^\top ,~(\mathbf x_1,\mathbf x_2)~{\rm is~a~leading~singular~ vector ~pair}&\\
~~~~~~~~~~~~~~~~~~~~{\rm~of~} \boldsymbol{Mat}( \Lambda^{k} + \tau\mathcal Y^{k}) & d~{\rm  odd},\\  
	\end{array}  
	\right.  \nonumber\\
	\mathcal Y^{k+1} &=& \arg\min_{\mathcal Y\in \mathbb S^{n^d}}L_\tau(\mathcal X^{k+1}, \mathcal Y,\Lambda^{k})\\
	&\Leftrightarrow &\langle -\mathcal A + \Lambda^{k} + \tau(\mathcal Y^{k+1}-\mathcal X^{k+1}), \mathcal Y-\mathcal Y^{k+1}\rangle\geq 0, ~\forall \mathcal Y\in \mathbb S^{n^d} ,  \nonumber\\
		&\Leftrightarrow &\mathcal Y^{k+1} = \tau^{-1}\boldsymbol{Sym}( \mathcal A - \Lambda^{k} + \tau\mathcal X^{k+1} ),\nonumber\\
	\Lambda^{k+1} &=& \Lambda^{k} - \tau(\mathcal X^{k+1} - \mathcal Y^{k+1}).\nonumber
	\end{eqnarray}
\end{boxedminipage}

Clearly, if $\lim_{k\rightarrow\infty} \mathcal X^k = \lim_{k\rightarrow \infty}\mathcal Y^k = \mathcal X^*=\mathcal Y^*$, then according to Theorem \ref{prop:equivalance}, the resulting tensor is of rank-1, which is feasible to  \eqref{prob:original4}.
\begin{remark}\label{rmk:sec:convergence:1}~\\
	$~~~~~\,\mathcal X$-subproblem: From the definition of $C$ and Remark  \ref{rmk:1}, it  amounts to computing   the leading eigenvalue/singular value of the matrix $\boldsymbol{ Mat }(\Lambda^k+\tau\mathcal Y^k)$. We dot not need to increase $d$ when $d$ is odd.
	
$\mathcal Y$-subproblem: The variational inequality   holds iff $\boldsymbol{ Sym }(-\mathcal A + \Lambda^k + \tau(\mathcal Y^{k+1} - \mathcal X^{k+1})) = 0$, which, together with that $\mathcal Y^{k+1}\in\mathbb S^{n^d}$ yields that $\mathcal Y^{k+1} = \tau^{-1}\boldsymbol{Sym}( \mathcal A - \Lambda^{k} + \tau\mathcal X^{k+1} )$.

The dominant computational complexity of each iterate of the ADMM is   the $\mathcal X$-subproblem, which     has the computational complexity $O(n^d)$ in theory,   which indicates that the algorithm should be efficient and scalable.

	By noticing the last relation, one obtains $\mathcal Y^{k+1} = \tau^{-1}\boldsymbol{Sym}(\mathcal A -\Lambda^{k+1} + \tau\mathcal Y^{k+1})$, which yields $\boldsymbol{Sym}(\Lambda^{k+1})=\mathcal A$ for $k=0,1,\ldots$; then it follows again from the $\mathcal Y$-subproblem that $\mathcal Y^{k+1} = \tau^{-1}\boldsymbol{Sym}(\mathcal A-\Lambda^{k}) + \boldsymbol{Sym}(\mathcal X^{k+1}) = \boldsymbol{Sym}(\mathcal X^{k+1})$. Based on these observations, the nonconvex ADMM can be simplied as:
 		
	 \begin{boxedminipage}{0.95\textwidth}\small
	 	Equivalent form of ADMM \eqref{alg:admm}:
	\begin{equation*}\small
\mathcal X^{k+1}\in\arg\min_{\mathcal X\in C} \langle -\Lambda^{k}-\tau\boldsymbol{Sym}(\mathcal X^{k}),\mathcal X\rangle,~~\Lambda^{k+1} = \Lambda^{k} - \tau(\mathcal X^{k+1}-\boldsymbol{Sym}(\mathcal X^{k+1})).
	\end{equation*}
\end{boxedminipage}
\end{remark}

\subsection{Convergence}\label{sec:conv}
The convergence of ADMM applied to nonconvex problems  was not well understood until recent years; see, e.g., \cite{li2015global,wang2019global,hong2016convergence}. Unfortunately, existing convergence results cannot be applied due to that the assumptions are not satisfied\footnote{To be more specific, using Proposition \ref{prop:sym_invariance}, \eqref{prob:relax3} can be equivalently written as $\min_{\mathcal X=\boldsymbol{ Sym }(\mathcal Y),\mathcal X\in\mathbb R^{n^d},\mathcal Y \in\mathbb R^{n^d}     } \langle -\mathcal A, \mathcal Y\rangle + I_C(\mathcal X)     $, where $I_C(\cdot)$ denotes the indicator function of the set $C$; while in the literature such as \cite{li2015global,wang2019global}, the convergence of ADMM applied to   problems of the form $\min_{A\mathbf x+B\mathbf y=0} f(\mathbf x) + g(\mathbf y)$ with $f$ being nonconvex and nonsmooth, and $g$ being nonconvex and smooth typically assumes that ${\rm Im  }(A) \subseteq {\rm Im }(B)$, where ${\rm Im}(\cdot)$ is the image of a matrix.   Clearly, in our setting, such assumption cannot be met, in that we have an  ``opposite'' situation ${\rm Im}(B)\subset {\rm Im}(A)$. }.  On the other hand, for nonconvex algorithms, commonly the best one can expect is the convergence to a stationary point. Interestingly,  we will show that  the algorithm is able to converge to a global minimizer to \eqref{prob:relax3}, namely, the leading eigenpair of $\mathcal A$ can be found by the nonconvex ADMM. 






%
%

\begin{theorem}[Global convergence to a global minimizer]\label{th:convergence_admm_even} Let $d$ be even.
Let $\{\mathcal X^{k},\mathcal Y^{k}, \Lambda^{k}   \}$ be generated by the nonconvex ADMM \eqref{alg:admm}. Assume that there exists a tuple $\{ \mathcal X^*,\mathcal Y^*,\Lambda^* \}$ satisfying the KKT system \eqref{eq:kkt_nonconvex_reformulation}; without loss of generality assume that $\boldsymbol{Mat}(\Lambda^*)$ is a symmetric matrix.
Assume that the leading eigenvalue of    $\boldsymbol{Mat}(\Lambda^*)$ is a simple root.

Then, if $\tau>0$ is chosen properly small so that \eqref{eq:proof:10} holds, the primal variable  $\{\mathcal X^{k},\mathcal Y^{k} \}$ converges to $\{\mathcal X^*,\mathcal Y^* \}$ which is a global optimizer to the original problem \eqref{prob:relax3} (or \eqref{prob:original4}).
\end{theorem}

Some remarks are presented first.

1. Let $\mathcal X^* = \mathbf x^{\circ^d}$. Then $\mathbf x$ is a leading eigenvalue of $\mathcal A$.

	2. Under the hypothesis,  Theorem \ref{th:global_solution} shows that $\{ \mathcal X^*,\mathcal Y^*,\Lambda^* \}$ satisfies \eqref{eq:kkt_nonconvex_reformulation_notau} as well. Concerning the hypothesis on  $\Lambda^*$, it is not easy to check it a prior, although in practice it is commonly   observed that such $\Lambda^*$ exists.  We will further study this issue in Section \ref{sec:uniqueness}. 

3. In practice, let $\{\overline{\mathcal X},\overline{\mathcal Y},\overline{\Lambda} \}$ be the output of the algorithm. Then we can use the optimality condition \eqref{eq:kkt_nonconvex_reformulation_notau} to determine if $\overline{\mathcal X}$ is optimal to \eqref{prob:original4}, namely, the algorithm has the ability to tell us that $\overline{\mathcal X}$ is optimal if  $ \langle -\overline{\Lambda}, \mathcal X-\overline{\mathcal X}\rangle \geq 0,\forall \mathcal X\in C$.   It should be also pointed out that, such inequality is only   sufficient   to determine the optimality: If the aforementioned inequality does not hold, then $\overline{\mathcal X}$ might still be optimal to \eqref{prob:original4}, which is often observed in the experiments. This is because even if the assumptions of Theorem \ref{th:convergence_admm_even} are met and $\mathcal X^k\rightarrow\mathcal X^*$, $\Lambda^k$ might converge to some point other than $\Lambda^*$ (the proof also does not tell that what $\lim_{k\rightarrow\infty}\Lambda^k$ is). Overall, \textit{the nonconvex ADMM can identify the optimal solution itself to a certain extent}.


4. Although the theorem requires a small $\tau$ to ensure the convergence in theory, it is   impractical if  $\tau$ is chosen small. Nevertheless, we have observed that for a large $\tau$, the algorithm still converges, and can find the global minimizer in a large chance. 

%
%
%
%

The following lemma is crucial for the convergence.
\begin{lemma}\label{lem:proof:1}
	Let $B\in \mathbb R^{m\times m}$ be a symmetric matrix with its   eigenvalues $\sigma_1,\ldots,\sigma_{m}$ being arranged in a descending order; let $\mathbf x_m$ be the normalized eigenvector   corresponding to $\sigma_{m}$. For any $\mathbf x\in \mathbb R^m,\mathbf \|\mathbf x\|=1 $, there holds
	\begin{equation*}
	      \setlength\abovedisplayskip{2pt}
	\setlength\abovedisplayshortskip{2pt}
	\setlength\belowdisplayskip{2pt}
	\setlength\belowdisplayshortskip{2pt}
	\langle B,\mathbf x\mathbf x^{\top} -  \mathbf x_m\mathbf x^{\top}_m \rangle\geq \frac{\sigma_{m-1} - \sigma_{m} }{2}\| \mathbf x\mathbf x^{\top} - \mathbf x_m\mathbf x^{\top}_m\|_F^2.
	\end{equation*}
\end{lemma}
\begin{proof}
	Let $B= U\Sigma U^{\top}$ be an EVD of $B$, with $U=[\mathbf u_1,\ldots,\mathbf u_m  ]\in \mathbb R^{m\times m}$ being an orthonormal matrix, $\Sigma ={\rm diag}(\sigma_1,\ldots,\sigma_{m-1}, \sigma_m  )\in\mathbb R^{m\times m}$ being a diagonal matrix, where $\sigma_i,1\leq i\leq m$ are arranged in a descending order. Then   any normalized $\mathbf x\in\mathbb R^m$ can be expressed as
	$\mathbf x = \sum^m_{i=1}\alpha_i \mathbf u_i,~{\rm with}~ \sum^m_{i=1}\alpha_i^2=1$.
	With the above expressions at hand, and noticing that 
	$
	\langle \mathbf x\mathbf x^{\top}, \mathbf x_m\mathbf x^{\top}_m\rangle = \alpha_m^2 
	$, we have
	\begin{small}
	\begin{eqnarray*}
		      \setlength\abovedisplayskip{1pt}
		\setlength\abovedisplayshortskip{1pt}
		\setlength\belowdisplayskip{1pt}
		\setlength\belowdisplayshortskip{1pt}
		\mathbf x^{\top}B\mathbf x - \mathbf x^{\top}_mB\mathbf x_m - \frac{\sigma_{m-1} - \sigma_{m} }{2}\| \mathbf x\mathbf x^{\top} - \mathbf x_m\mathbf x^{\top}_m\|_F^2
		&=& \sum^m_{i=1}\alpha_i^2\sigma_i - \sigma_m - (\sigma_{m-1} - \sigma_m)(1-\alpha_m^2) \\
		&=& \sum^{m-1}_{i=1}\alpha_i^2\sigma_i -(1-\alpha_m^2)\sigma_{m-1} \\
		&=& \sum^{m-1}_{i=1}\alpha_i^2(\sigma_i-\sigma_{m-1}) \geq 0,
	\end{eqnarray*}
\end{small}
	where the last equality follows from $\sum^m_{i=1}\alpha_i^2=1$. The proof has been completed.
\end{proof}

\begin{proof}[Proof of Theorem \ref{th:convergence_admm_even}]  
Let   $\sigma_{n^{d/2}}$ and $\sigma_{n^{d/2-1}}$ respectively denote the smallest and the second smallest eigenvalues of $-\boldsymbol{Mat}(\Lambda^*)$. It follows from   the hypothesis that $\sigma_{n^{d/2}-1 } > \sigma_{n^{d/2} } $. Assume that $\tau$ satisfies $0<\tau<   \beta=\sigma_{n^{d/2}-1 } - \sigma_{n^{d/2} } $.
According to Theorem \ref{prop:proof:1},  $\mathcal X^*\in\arg\min_{\mathcal X\in C}\langle-\Lambda^*,\mathcal X\rangle = \langle-\boldsymbol{Mat}(\Lambda^*),\boldsymbol{Mat}(\mathcal X) \rangle$ under the assumptions.	By   Lemma \ref{lem:proof:1} and recalling the definition of $C$ and Remark \ref{rmk:1}, we have  
	\begin{equation}\label{eq:proof:2}
			      \setlength\abovedisplayskip{2pt}
	\setlength\abovedisplayshortskip{2pt}
	\setlength\belowdisplayskip{2pt}
	\setlength\belowdisplayshortskip{2pt}
	\langle-\Lambda^* ,\mathcal X^{k+1} - \mathcal X^{*}\rangle \geq  \frac{\beta}{2} \|\mathcal X^{k+1}-\mathcal X^*\|_F^2.
	\end{equation}
	On the other hand,	from   the first inequality of \eqref{alg:admm} with  $\mathcal X:=\mathcal X^*$, we have
	\begin{eqnarray}\langle  - \Lambda^{k} - \tau   \mathcal Y^{k} +\tau \mathcal X^{k+1}   , \mathcal X^* - \mathcal X^{k+1}\rangle &=& \langle  - \Lambda^{k} - \tau   \mathcal Y^{k},\mathcal X^* - \mathcal X^{k+1}\rangle  +\tau\langle \mathcal X^{k+1}   , \mathcal X^* - \mathcal X^{k+1}\rangle \nonumber\\
	 &\geq& \tau \langle \mathcal X^{k+1} ,\mathcal X^*-\mathcal X^{k+1}\rangle\nonumber\\
	&=&-\frac{\tau }{2}\|\ \mathcal X^{k+1} - \mathcal X^*\|_F^2, \label{eq:proof:1}
	\end{eqnarray}
where the last equality holds because $\mathcal X^{k+1}$ and $\mathcal X^*$ are normalized. Using the last equality in \eqref{alg:admm}, $-\Lambda^{k} + \tau\mathcal X^{k+1}$ can be replaced by $-\Lambda^{k+1}+ \tau\mathcal Y^{k+1}$ in the left-hand side of \eqref{eq:proof:1}. Adding \eqref{eq:proof:1} together with \eqref{eq:proof:2} gives
\begin{small}
	\begin{eqnarray}\label{eq:proof:11}
	      \setlength\abovedisplayskip{2pt}
	\setlength\abovedisplayshortskip{2pt}
	\setlength\belowdisplayskip{2pt}
	\setlength\belowdisplayshortskip{2pt}
	\langle \Lambda^{k+1}-\Lambda^*, \mathcal X^{k+1}-\mathcal X^*\rangle + \tau\langle \mathcal Y^{(k )}-\mathcal Y^{k+1}, \mathcal X^{k+1}-\mathcal X^*\rangle&\geq&  
     {  \frac{\beta-\tau}{2} } \|\ \mathcal X^{k+1} - \mathcal X^*\|_F^2.
	\end{eqnarray}
	\end{small}
	Next, since $\tau<\beta$, Theorem \ref{prop:proof:1} implies that $\{\mathcal X^*,\mathcal Y^*,\Lambda^* \}$ meets the optimality condition \eqref{eq:kkt_nonconvex_reformulation_notau}, where Theorem \ref{th:global_solution} asserts that $\boldsymbol{Sym}(\Lambda^*)=\mathcal A$. On the other hand, in Remark \ref{rmk:sec:convergence:1} we have discussed that $\boldsymbol{Sym}(\Lambda^{k+1})=\mathcal A$ for $k=0,1,\ldots$. It then follows from $\mathcal Y^*,\mathcal Y^{k}\in\mathbb S^{n^d}$ and Proposition \ref{prop:sym_invariance} that
	\begin{equation}\label{eq:proof:12}
			      \setlength\abovedisplayskip{2pt}
	\setlength\abovedisplayshortskip{2pt}
	\setlength\belowdisplayskip{2pt}
	\setlength\belowdisplayshortskip{2pt}
	\langle \Lambda^{k+1} - \Lambda^*,\mathcal Y^{*}-\mathcal Y^{k+1} \rangle =\langle \boldsymbol{Sym}(\Lambda^{k+1} - \Lambda^*), \mathcal Y^{*}-\mathcal Y^{k+1}\rangle= 0;
	\end{equation}
\begin{equation}\label{eq:proof:13}
		      \setlength\abovedisplayskip{2pt}
\setlength\abovedisplayshortskip{2pt}
\setlength\belowdisplayskip{2pt}
\setlength\belowdisplayshortskip{2pt}
\langle \Lambda^{k+1}-\Lambda^{k},\mathcal Y^{k} - \mathcal Y^{k+1} \rangle =\langle\boldsymbol{Sym}(\Lambda^{k+1}-\Lambda^{k}),\mathcal Y^{k} - \mathcal Y^{k+1} \rangle= 0.
\end{equation}
Combining \eqref{eq:proof:11}, \eqref{eq:proof:12} and \eqref{eq:proof:13}, we have
\begin{eqnarray}
&& (\frac{\beta- \tau}{2}) \|\mathcal X^{k+1}-\mathcal X^*\|_F^2\nonumber\\
\leq && \langle \Lambda^{k+1}-\Lambda^*,\mathcal X^{k+1}-\mathcal X^*\rangle + \langle\Lambda^{k+1}-\Lambda^*, \mathcal Y^*-\mathcal Y^{k+1}\rangle\nonumber\\
 && ~~~~+\langle \mathcal Y^{k}-\mathcal Y^{k+1},\tau(\mathcal X^{k+1}-\mathcal X^*)\rangle + \langle \mathcal Y^{k} - \mathcal Y^{k+1},\Lambda^{k+1}-\Lambda^{k}\rangle \nonumber\\
 =&& \tau^{-1}\langle \Lambda^{k+1}-\Lambda^*, \Lambda^{k}-\Lambda^{k+1}\rangle + \tau\langle \mathcal Y^{k}-\mathcal Y^{k+1},\mathcal Y^{k+1}-\mathcal Y^*\rangle.\label{eq:proof:14}
\end{eqnarray}


On the other hand, we have	
\begin{small}
\begin{eqnarray}
\tau \|\mathcal Y^{k+1} -\mathcal Y^*\|_F^2 &=& \tau  \|\mathcal Y^{k} - \mathcal Y^*\|_F^2 - 2\tau \langle \mathcal Y^{k}-\mathcal Y^*,\mathcal Y^{k}-\mathcal Y^{k+1}\rangle + \tau \|\mathcal Y^{k}-\mathcal Y^{k+1}\|_F^2\nonumber\\
&=& \tau \|\mathcal Y^{k}-\mathcal Y^*\|_F^2 - \tau \|\mathcal Y^{k+1}-\mathcal Y^{k}\|_F^2 - 2\tau \langle \mathcal Y^{k+1}-\mathcal Y^{*},\mathcal Y^{k}-\mathcal Y^{k+1}\rangle.\label{eq:proof:7}
\end{eqnarray}	
\end{small}
Similarly, 
\begin{eqnarray}
\tau^{-1}\|\Lambda^{k+1}-\Lambda^*\|_F^2 &=& \tau^{-1}\|\Lambda^{k}-\Lambda^*\|_F^2 - \tau^{-1}\|\Lambda^{k+1}-\Lambda^{k}\|_F^2   \nonumber\\
&&~~~~~~~~-2\tau^{-1}\langle \Lambda^{k+1}-\Lambda^*,\Lambda^{k}-\Lambda^{k+1}\rangle. \label{eq:proof:8}
\end{eqnarray}
	Summing \eqref{eq:proof:7} and \eqref{eq:proof:8} together, and using \eqref{eq:proof:14}, we obtain
	\begin{eqnarray}
	\tau \|\mathcal Y^{k+1}-\mathcal Y^{*}\|_F^2 + \tau^{-1}\|\Lambda^{k+1}-\Lambda^*\|_F^2 &\leq& \tau \left( \|\mathcal Y^{k}-\mathcal Y^*\|_F^2 - \|\mathcal Y^{k+1}-\mathcal Y^{k}\|_F^2     \right)\nonumber\\
	&& + \tau^{-1}\left(\|\Lambda^{k}-\Lambda^*\|_F^2 -  \|\Lambda^{k+1}-\Lambda^{k}\|_F^2\right)\nonumber\\
	&& -  (\frac{\beta- \tau}{2}) \|\mathcal X^{k+1}-\mathcal X^*\|_F^2   .\label{eq:proof:30}
	\end{eqnarray}
	The above inequality shows that $\{ \mathcal Y^{k},\Lambda^{k} \}$ is bounded. On the other side, rearranging and combining terms yields
	\begin{eqnarray*}
&& \tau \|\mathcal Y^{k+1}-\mathcal Y^{k}\|_F^2 + \tau^{-1}\|\Lambda^{k+1}-\Lambda^{k}\|_F^2 + (\frac{\beta- \tau}{2}) \|\mathcal X^{k+1}-\mathcal X^*\|_F^2\\
 &\leq& \tau \left(  \|\mathcal Y^{k}-\mathcal Y^*\|_F^2 -\|\mathcal Y^{k+1}-\mathcal Y^{*}\|_F^2   \right) 
  + \tau^{-1}\left( \|\Lambda^{k}-\Lambda^*\|_F^2 - \|\Lambda^{k+1}-\Lambda^*\|_F^2   \right) .
	\end{eqnarray*}
Summing the above inequality from $0$ to infinity, we get
\begin{small}
\begin{eqnarray*}\label{eq:proof:50}
\sum^{\infty}_{k=0}\left(\tau \|\mathcal Y^{k+1}-\mathcal Y^{k}\|_F^2 + \tau^{-1}\|\Lambda^{k+1}-\Lambda^{k}\|_F^2 + (\frac{\beta- \tau}{2}) \|\mathcal X^{k+1}-\mathcal X^*\|_F^2 \right)  <  +\infty.
\end{eqnarray*}	
\end{small}
Thus we have
$$\mathcal Y^{k+1}-\mathcal Y^{k}\rightarrow 0, \Lambda^{k+1}-\Lambda^{k}\rightarrow 0, \mathcal X^{k+1}-\mathcal X^*\rightarrow 0, ~{\rm}~k\rightarrow\infty,$$
and so $\lim_{k\rightarrow \infty}\mathcal Y^{k}=\lim_{k\rightarrow\infty}\mathcal X^{k}=\mathcal X^*$. Since $\mathcal X^*$ satisfies the   system \eqref{eq:kkt_nonconvex_reformulation}, and $\tau$ is chosen as \eqref{eq:proof:10}, by Theorem \ref{prop:proof:1}, $\mathcal X^*$ satisfies \eqref{eq:kkt_nonconvex_reformulation_notau}, which together with Theorem \ref{th:global_solution} shows that     $\mathcal X^*$ ($\mathcal Y^*$)  is   a global optimizer of the original problem \eqref{prob:original4}. The proof has been completed. 
\end{proof}

\paragraph{Convergence with an arbitrary $\tau>0$}
It is also possible to set an arbitrary $\tau>0$ in theory. This can be done by taking a more careful estimation to the left-hand side of \eqref{eq:proof:1}. Let $\sigma_{i}(\cdot)$ denote  the $i$-th eigenvalue of a matrix arranged in a descending order. Note that $\mathcal X^{k+1}$ corresponds to     $\sigma_{n^{d/2}}(-\boldsymbol{ Mat }(\Lambda^{k} +\tau\mathcal Y^{k}))$, and also corresponds to the eigenvalue $\sigma_{n^{d/2}}(-\boldsymbol{ Mat }(\Lambda^{k} +\tau\mathcal Y^{k}))+\tau$ of $-\boldsymbol{ Mat }(\Lambda^{k}+\tau\mathcal Y^{k} -\tau\mathcal X^{k+1})$, which means that if $\sigma_{n^{d/2}}(-\boldsymbol{ Mat }(\Lambda^{k} +\tau\mathcal Y^{k}))+\tau$ is the smallest eigenvalue of $-\boldsymbol{ Mat }(\Lambda^{k}+\tau\mathcal Y^{k} -\tau\mathcal X^{k+1})$, then the left-hand side of \eqref{eq:proof:1} is nonnegative, and the proof of Theorem \ref{th:convergence_admm_even} carries over. In the following, we consider one of such cases. To this end, for $\{ \mathcal Y^*,\Lambda^*  \}\in\mathbb S^{n^d}\times\mathbb R^{n^d}$ we   define
$$\mathbb B_\tau( \{\mathcal Y^*,\Lambda^* \},\mu ):=  \{ \{\mathcal Y,\Lambda  \}\mid \tau\|\mathcal Y-\mathcal Y^*\|_F^2 + \tau^{-1}\|\Lambda -\Lambda^*\|_F^2 \leq \mu^2   \}.$$
Assume that the smallest eigenvalue of $-\boldsymbol{ Mat }(\Lambda^*)$ is a simple root. Let $\overline\mu>0$ be such that
\begin{small}
 \begin{equation}\label{eq:proof:20}
 \begin{split}
&\overline\mu=\arg\sup_{\mu>0}\max_{ \{\mathcal Y,\Lambda \}\in \mathbb B_\tau(\{\mathcal Y^*,\Lambda^* \}, \mu   ) }\|  \boldsymbol{ Mat }(\Lambda-\Lambda^* + \tau\mathcal Y  - \tau\mathcal Y^*)\|_2  \\
&~~~~~~~~~~~~~~~~~~~~~~~~~~~~~~~~~~~~~~~~\leq \frac{\sigma_{n^{d/2-1} }(-\boldsymbol{ Mat }(\Lambda^*)) - \sigma_{n^{d/2} }(-\boldsymbol{ Mat }(\Lambda^*)) }{4}.  
\end{split}
\end{equation}
\end{small}
\begin{theorem}[Local convergence to a global minimizer with arbitrary $\tau>0$]\label{th:anytau}
Let $d$ be even.	Let $\{\mathcal X^*,\mathcal Y^*,\Lambda^* \}$ satisfy \eqref{eq:kkt_nonconvex_reformulation_notau}. Assume without loss of generality  that $\boldsymbol{Mat}(\Lambda^*)$ is symmetric.  Assume that $-\boldsymbol{ Mat }(\Lambda^*)$ has a simple smallest eigenvalue. Let $\mathbb B_\tau(\{\mathcal Y^*,\Lambda^* \},\overline\mu   )$ be defined as above with $\tau>0$. If $\{\mathcal Y^0,\Lambda^0  \}\in \mathbb B_\tau(\{\mathcal Y^*,\Lambda^* \},\overline\mu   )$,
	 then $\{\mathcal X^k,\mathcal Y^k  \}$ converges to $\{\mathcal X^*,\mathcal Y^* \}$ which is a global optimizer to \eqref{prob:original4}.
\end{theorem}
\begin{proof}
To make the proof of Theorem \ref{th:convergence_admm_even} carries over, it suffices   to show that 
	\begin{equation}\label{eq:proof:31}
\langle  - \Lambda^{k} - \tau   \mathcal Y^{k} +\tau \mathcal X^{k+ 1}   , \mathcal X  - \mathcal X^{k+ 1}\rangle \geq 0,~ \forall \mathcal X\in C 
\end{equation}
holds   in the current setting.
We first consider $k=0$, and	denote $\mathcal E:= -\Lambda^0-\tau\mathcal Y^0 - (-\Lambda^*-\tau\mathcal Y^*) $. According to Weyl's inequality, it holds that $|\sigma_{i}(-\boldsymbol{ Mat }(\Lambda^0+\tau\mathcal Y^0)) - \sigma_{i}(-\boldsymbol{ Mat }(\Lambda^*+\tau\mathcal Y^*)  )   | \leq \|\boldsymbol{ Mat }(\mathcal E)\|_2$, which together with \eqref{eq:proof:20} and the definition of $\mathcal Y^*$ yields
   \begin{small}
	\begin{eqnarray}\label{eq:proof:19}
\langle  - \Lambda^{0} - \tau   \mathcal Y^{0} +\tau \mathcal X^{ 1}   ,   \mathcal X^{ 1}\rangle &=&	\sigma_{n^{d/2}}(-\boldsymbol{ Mat }(\Lambda^0+\tau\mathcal Y^0))+\tau\\
  &\leq & \sigma_{n^{d/2} }(-\boldsymbol{ Mat }(\Lambda^*+\tau\mathcal Y^*)  ) + \|\boldsymbol{ Mat }(\mathcal E)\|_2+\tau\nonumber\\
	&=&\sigma_{n^{d/2} }(-\boldsymbol{Mat}(\Lambda^*)) - \tau + \|\boldsymbol{ Mat }(\mathcal E)\|_2 + \tau\nonumber\\
	&\leq&  \frac{\sigma_{n^{d/2-1} }(-\boldsymbol{ Mat }(\Lambda^*)) + 3\sigma_{n^{d/2} }(-\boldsymbol{ Mat }(\Lambda^*)) }{4} . \nonumber
	\end{eqnarray}
	\end{small}
	By the definition of $\mathcal X^1$ and the structure of $-\Lambda^0-\tau\mathcal Y^0+\tau\mathcal X^1$, the right-hand side of \eqref{eq:proof:19} is an   eigenvalue of $-\boldsymbol{ Mat }(\Lambda^0+\tau\mathcal Y^0-\tau\mathcal X^1)$. We  show that it is   the smallest one, i.e., \eqref{eq:proof:31} holds when $k=0$. If this is not true, then 
	\begin{equation}\label{eq:proof:39} 
		\sigma_{n^{d/2}}(-\boldsymbol{ Mat }(\Lambda^0+\tau\mathcal Y^0-\tau\mathcal X^1)) =\sigma_{n^{d/2}-1  }(-\boldsymbol{ Mat }(\Lambda^0+\tau\mathcal Y^0))\footnote{The term $\tau \boldsymbol{Mat}(\mathcal X^1)$ only shifts the smallest eigenvalue of $-\boldsymbol{Mat}(\Lambda^0+\tau\mathcal Y^0 )$. If the shift is large, then $\sigma_{n^{d/2}-1  }(-\boldsymbol{ Mat }(\Lambda^0+\tau\mathcal Y^0))$ might turn to be the smallest eigenvalue of $  -\boldsymbol{ Mat }(\Lambda^0+\tau\mathcal Y^0-\tau\mathcal X^1) $.}.\end{equation}
	 However, it follows again from the Weyl's inequality and \eqref{eq:proof:20} that
	\begin{eqnarray}\label{eq:proof:22}
	\sigma_{n^{d/2}-1  }(-\boldsymbol{ Mat }(\Lambda^0+\tau\mathcal Y^0)  ) &\geq& \sigma_{n^{d/2 }-1}(-\boldsymbol{ Mat }(\Lambda^*+\tau\mathcal Y^*)  ) - \|\boldsymbol{ Mat }(\mathcal E)\|_2\nonumber\\
	&=&\sigma_{n^{d/2 }-1}(-\boldsymbol{ Mat }(\Lambda^* )  ) - \|\boldsymbol{ Mat }(\mathcal E)\|_2\nonumber\\
	&\geq&  \frac{3\sigma_{n^{d/2-1} }(-\boldsymbol{ Mat }(\Lambda^*)) + \sigma_{n^{d/2} }(-\boldsymbol{ Mat }(\Lambda^*)) }{4},
	\end{eqnarray}
where the first equality is due to the definition of $\mathcal Y^*$.  \eqref{eq:proof:22} together with	  
	  \eqref{eq:proof:19}  shows that $\sigma_{n^{d/2}-1  }(-\boldsymbol{ Mat }(\Lambda^0+\tau\mathcal Y^0)  )$ is not the smallest eigenvalue of  $-\boldsymbol{ Mat }(\Lambda^0+\tau\mathcal Y^0-\tau\mathcal X^1)$,  which contradicts with \eqref{eq:proof:39}. Hence $\sigma_{n^{d/2}}(-\boldsymbol{ Mat }(\Lambda^0+\tau\mathcal Y^0)) + \tau$ is the smallest eigenvalue of $-\boldsymbol{ Mat }(\Lambda^0+\tau\mathcal Y^0-\tau\mathcal X^1)$, and so \eqref{eq:proof:31} holds when $k=0$. \eqref{eq:proof:31} together with \eqref{eq:proof:2} yields
	  \begin{small}
		\begin{eqnarray*}\label{eq:proof:35}
	\langle \Lambda^{k+1}-\Lambda^*, \mathcal X^{k+1}-\mathcal X^*\rangle + \tau\langle \mathcal Y^{ k  }-\mathcal Y^{k+1}, \mathcal X^{k+1}-\mathcal X^*\rangle&\geq&  
 \frac{\beta }{2} \|\ \mathcal X^{k+1} - \mathcal X^*\|_F^2,
	\end{eqnarray*}
	\end{small}
	which is similar to \eqref{eq:proof:11}. 
	Carrying on similarly, we see that \eqref{eq:proof:30} (except the coefficient of the last term of the right-hand side being $-\frac{\beta}{2}$) is valid in the current setting where $k=0$, which then implies that $\{\mathcal Y^1,\Lambda^1  \} \in \mathbb R_\tau(\{\mathcal Y^*,\Lambda^* \},\overline\mu )$ as well. Inductively, we are able to show that  \eqref{eq:proof:31} holds for all $k$ and $\{\mathcal Y^k,\Lambda^k\}\in\mathbb R_\tau(\{\mathcal Y^*,\Lambda^* \},\overline\mu )$, therefore resulting in the validness of \eqref{eq:proof:30} for all $k$, namely, $\{\mathcal X^k,\mathcal Y^k \}\rightarrow \{ \mathcal X^*,\mathcal Y^* \}$, which is a global optimizer to \eqref{prob:original4}, as desired.
\end{proof}

\paragraph{Converging to other  eigenvectors}

In any case, we have observed   extensively   that $\mathcal X^k-\mathcal Y^k\rightarrow 0$. With this phenomenon, the following results readily follow.
\begin{proposition}\label{prop:any}
Let $\{\mathcal X^{k},\mathcal Y^{k}, \Lambda^{k}   \}$ be generated by the nonconvex ADMM \eqref{alg:admm}. If $\{\Lambda^k \}$ is bounded and $\mathcal X^k-\mathcal Y^k\rightarrow 0$, then every limit point of $\{\mathcal X^{k},\mathcal Y^{k}, \Lambda^{k}   \}$ yields an eigenvector of $\mathcal A$.
\end{proposition}
\begin{proof} The proof is a standard routine.
Since $\mathcal X^k\in C$,	the boundedness of $\{\mathcal X^k \}$ and the coercivity of $L_\tau$ with respect to  $\mathcal Y$ shows that $\{\mathcal Y^k \}$ is also bounded, and so $\{L_\tau(\mathcal X^k,\mathcal Y^k,\Lambda^k )\}$ is bounded from below. This then together with the strong convexity of  $L_\tau$ and the definition of $\mathcal Y^k$ implies that $L_\tau(\mathcal X^{k+1},\mathcal Y^k,\Lambda^k)- L_\tau(\mathcal X^{k+1},\mathcal Y^{k+1},\Lambda^k) \geq \frac{\tau}{2}\|\mathcal Y^{k }-\mathcal Y^{k+1}\|_F^2$ and hence $\mathcal Y^k-\mathcal Y^{k+1}\rightarrow 0$. This combines with $\mathcal X^k-\mathcal Y^k\rightarrow 0$ shows that $\mathcal X^k-\mathcal X^{k+1}\rightarrow 0$. Let $\{\mathcal X^{k_l},\mathcal Y^{k_l},\Lambda^{k_l}  \}$ be a subsequence  converging to $\{\mathcal X^*,\mathcal Y^*,\Lambda^* \}$ as $l\rightarrow \infty$. Then $\{\mathcal X^{k_l+1},\mathcal Y^{k_l+1},\Lambda^{k_l+1}  \}$ possesses the same limit. As a consequence, taking the limit into \eqref{alg:admm} with respect to $l$ yields that $\{\mathcal X^*,\mathcal Y^*,\Lambda^* \}$ satisfies \eqref{eq:kkt_nonconvex_reformulation}, which, by Proposition \ref{prop:kkt_withtau}, yields an eigenpair of $\mathcal A$.
\end{proof}

\begin{remark}
 Theorems \ref{th:convergence_admm_even}, \ref{th:anytau} and Proposition \ref{prop:any} carry over to the odd order case analogously.
\end{remark}

\subsection{The hypothesis of Theorem \ref{th:convergence_admm_even}}\label{sec:uniqueness}
 
Recall that we assume the existence of $\{\mathcal X^*,\mathcal Y^*,\Lambda^*  \}$ to \eqref{eq:kkt_nonconvex_reformulation} with the leading eigenvalue of $\boldsymbol{ Mat }(\Lambda^*)$ being simple. Restricting to even order cases, we study the hypothesis from the     the convex relaxation \eqref{prob:relax_convex} and its dual. Similar to \eqref{prob:relax3}, we introduce $\mathcal Y$ and  equivalently rewrite  \eqref{prob:relax_convex} as
\begin{equation}	      \setlength\abovedisplayskip{2pt}
\setlength\abovedisplayshortskip{2pt}
\setlength\belowdisplayskip{2pt}
\setlength\belowdisplayshortskip{2pt}\tag{\ref{prob:relax_convex}$^\prime$}
\label{prob:relax_convex_y}
(R)~~\max~\langle   \mathcal A,\mathcal Y\rangle~{\rm s.t.}~ \mathcal X=\mathcal Y, \mathcal X\in C_R,~\mathcal Y\in\mathbb S^{n^d},
\end{equation}
where
$$
C_R:=  
\{\mathcal X\mid \boldsymbol{ Mat }(\mathcal X)\in \mathbb S^{ n^{   d/2   }\times n^{   d/2   }  }, ~{\rm tr}(\boldsymbol{ Mat }(\mathcal X))=1,~\boldsymbol{ Mat }(\mathcal X)\succeq 0   \}.
$$
The Lagrangian function for \eqref{prob:relax3} is  
$L(\mathcal X,\mathcal Y,\Lambda) := \langle  \mathcal A,\mathcal Y\rangle +\langle \Lambda,\mathcal X-\mathcal Y\rangle $, and
the KKT system for \eqref{prob:relax_convex_y}  is (in the simplified form, similar to \eqref{eq:kkt_nonconvex_reformulation_notau_variant}):
\begin{equation}
\label{eq:kkt_nonconvex_reformulation_notau_variant_relax}
  \mathcal X^*\in\arg\max_{\mathcal X\in C_R}\nolimits \langle  \Lambda^*,\mathcal X\rangle,~\mathcal X^*\in C_R\cap \mathbb S^{n^d};~\boldsymbol{ Sym}(\Lambda^*) = \mathcal A,~\boldsymbol{ Mat }(\Lambda^*)\in\mathbb S^{n^{d/2}\times n^{d/2} }.
\end{equation}
Since \eqref{prob:relax_convex_y} is a linear SDP, $\{\mathcal X^*,\Lambda^*  \}$ to the above system exists. 
It can also be verified that the dual of \eqref{prob:relax_convex_y} is
\begin{equation}\label{prob:dual1}
(D)~~	\min  ~\sigma_{\max}(\boldsymbol{Mat}(\Lambda))~~{\rm s.t.}~\boldsymbol{Sym}(\Lambda)=\mathcal A,~\boldsymbol{ Mat }(\Lambda)\in\mathbb S^{ n^{d/2}\times n^{d/2} },
\end{equation}
where $\sigma_{\max}(\cdot)$ denotes the leading eigenvalue. Note that \eqref{prob:dual1} is also an SDP by replacing the objective by a new variable $\sigma$ and appending a new constraint $\sigma I \succeq \boldsymbol{ Mat }(\Lambda)$. Since Slater's condition holds, there is no duality gap between \eqref{prob:dual1} and \eqref{prob:relax_convex_y}.    $\{\mathcal X^*,\Lambda^* \}$ of \eqref{eq:kkt_nonconvex_reformulation_notau_variant_relax} gives a pair of optimizers to \eqref{prob:relax_convex_y} and \eqref{prob:dual1}.  The following proposition shows that we can study  the hypothesis on $\Lambda^*$ from the dual problem.
\begin{proposition}\label{prob:sec:hypo:1}
	Let $\{\mathcal X^*,\Lambda^*  \}$ satisfy \eqref{eq:kkt_nonconvex_reformulation_notau_variant_relax}, with the leading eigenvalue of $\boldsymbol{ Mat }(\Lambda^*)$ being simple. Then $\mathcal X^*\in C$, i.e., $\{\mathcal X^*,\Lambda^* \}$ satisfies \eqref{eq:kkt_nonconvex_reformulation_notau_variant}. 
	
	On the contrary, if  $\{\mathcal X^*,\Lambda^* \}$ satisfies \eqref{eq:kkt_nonconvex_reformulation_notau_variant} with the leading eigenvalue of $\boldsymbol{ Mat }(\Lambda^*)$ being simple, then $\{\mathcal X^*,\Lambda^* \}$ also satisfies \eqref{eq:kkt_nonconvex_reformulation_notau_variant_relax}.
\end{proposition}
\begin{proof}
Let $t$ be the multiplicity of $\sigma_{\max}(\boldsymbol{ Mat }(\Lambda^*))$ and  $\{\mathbf z_1,\ldots,\mathbf z_t  \}$ be its corresponding orthonormal leading eigenvectors. Then the solutions to $\max_{\mathcal X\in C_R}\nolimits \langle \Lambda^*,\mathcal X\rangle $ can be characterized as $$\small\arg\max_{\mathcal X\in C_R}\nolimits \langle  \Lambda^*,\mathcal X\rangle = \{ \mathcal X\mid \boldsymbol{ Mat }(\mathcal X) = \sum^t_{i=1}\nolimits\alpha_i\mathbf z_i\mathbf z_i^\top,~\forall \sum^t_{i=1}\nolimits\alpha_i=1,\alpha_i\geq 0,1\leq i\leq t    \},$$ from which we see that when $t=1$, $\boldsymbol{ Mat }(\mathcal X^*)$ is rank-$1$. This together with $\mathcal X^*\in\mathbb S^{n^d}$ and Theorem \ref{prop:equivalance} shows that $\mathcal X^*\in C$, and hence $\{\mathcal X^*,\Lambda^*  \}$ satisfies \eqref{eq:kkt_nonconvex_reformulation_notau_variant}. The contrary part is clear.
\end{proof}

Denote
$ {\rm mult(\cdot)}$
as
 the number of linearly independent  eigenvectors corresponding to the leading  eigenvalue of a  matrix or a tensor.
 Let $V_O$ and $V_D$ respectively denote the optimal values of \eqref{prob:original4} and \eqref{prob:dual1}. Then we have:
\begin{proposition}\label{prop:cond:1}
Assume     that ${\rm mult}(\mathcal A)=t$. Let $\Lambda^*$ be optimal to \eqref{prob:dual1}. Then   $V_O = V_D$ iff the leading eigenspace of $\boldsymbol{Mat}(\Lambda^*)$ contains exactly $t$ linearly independent vectors of the form $\mathbf x_i^{\otimes^{d/2}}\in\mathbb R^{n^{d/2}}$, $\|\mathbf x_i\|=1,~1\leq i\leq t $; moreover, the $\mathbf x_i$'s are exactly    the leading  eigenvectors of $\mathcal A$.
 \end{proposition}
\begin{proof}
 Necessity: Assume that $V_O=V_D$; then $  V_O =\sigma_{\max}(\boldsymbol{ Mat }(\Lambda^*))$.
 Denote $\mathbf x_i, 1\leq i \leq t $ as the linearly independent leading  eigenvectors of $\mathcal A$. Then it holds that
 $$\langle \mathbf x_i^{ \circ^d },\Lambda^*\rangle = \langle \mathbf x_i^{\circ^d},\boldsymbol{Sym}(\Lambda^*)\rangle = \langle \mathbf x_i^{\circ^d},\mathcal A\rangle =V_O=\sigma_{\max}(\boldsymbol{ Mat }(\Lambda^*)),$$
 where the first equality follows from Proposition \ref{prop:sym_invariance}. This implies that $\mathbf x_i^{\otimes^{d/2} }$ is a leading eigenvector of $\boldsymbol{Mat}(\Lambda^*)$. On the other hand, the above relation also tells us that every vector of the form $\mathbf x^{\otimes^{d/2}}$ in the leading eigenspace of $\boldsymbol{Mat}(\Lambda^*)$ also contributes a leading  eigenvector to $\mathcal A$. 

 Sufficiency: We have $V_D = \langle \mathbf x_i^{\circ^{d}},\Lambda\rangle = \langle\mathbf x_i^{\circ^d},\mathcal A\rangle \leq V_O$, which together with $V_D\geq V_O$   yields $V_D=V_O$.
\end{proof}
\begin{proof}[Proof of Proposition \ref{prop:admm:1}] If the SDP relaxation is tight, then $V_O=V_D$. According to Proposition \ref{prop:cond:1}, let $\mathcal Y^*=\mathcal X^*=\mathbf x_1^{\circ^d}$, where $\mathbf x_1^{\otimes^{d/2}}$ is a leading eigenvector of $\boldsymbol{ Mat }(\Lambda^*)$, as the notations in Proposition \ref{prop:cond:1}. Then $\{ \mathcal X^*,\mathcal Y^*,\Lambda^* \}$ is a solution to \eqref{eq:kkt_nonconvex_reformulation_notau}. 
	
	Assume that $\{ \mathcal X^*,\mathcal Y^*,\Lambda^* \}$ is a solution to \eqref{eq:kkt_nonconvex_reformulation_notau}; then it can be verified that $\{\mathcal X^*,\Lambda^*  \}$ is also a solution for the system \eqref{eq:kkt_nonconvex_reformulation_notau_variant_relax}, which shows that the SDP relaxation is tight.
	\end{proof}
	
\begin{remark}    
	Proposition \ref{prop:cond:1} also implies that ${\rm mult}(\mathcal A)=1$ is a necessary condition for the simplicity of the leading eigenvalue of $\boldsymbol{ Mat }(\Lambda^*)$.

\end{remark}

Define two sets as follows
\begin{small}
	\begin{eqnarray*}
		&&\boldsymbol{ \mathcal A}:=\{\mathcal A\in\mathbb S^{n^d} \mid \{\mathcal X^*,\mathcal Y^*,\Lambda^*\}~{\rm to~\eqref{eq:kkt_nonconvex_reformulation_notau}~exists } \},\\
		&&\boldsymbol{ \mathcal A}^+:= \{\mathcal A\in\boldsymbol{ \mathcal A }\mid {\exists \Lambda^*\rm ~with~the~leading~eigenvalue~of}~\boldsymbol{ Mat }(\Lambda^*)~{\rm being~simple} \}.
	\end{eqnarray*}
\end{small}
 It would be interesting to study their relations. When $d=2$, $\boldsymbol{\mathcal A}$ is exactly the set of all symmetric matrices; see Remark \ref{rmk:2}, and it is well known that $\boldsymbol{ \mathcal A }^+$ is an open and dense set in $\boldsymbol{ \mathcal A }$, i.e., $\boldsymbol{ \mathcal A }^+$ is a generic phenomenon. 
  However, when $d>2$, as $\boldsymbol{ Mat }(\Lambda^*)$ is related to an optimization problem \eqref{prob:dual1}, it is not clear whether such a phenomenon holds. 
  We have the following two results instead:
	\begin{theorem}\label{prop:dense}
	$\boldsymbol{ \mathcal A }^+$ is dense in $\boldsymbol{ \mathcal A }$.
\end{theorem}
\begin{proof}
	It suffices to show that	for any $\mathcal A\in\boldsymbol{ \mathcal A }$ and any $\epsilon>0$, there exists a $\mathcal A_\epsilon\in\boldsymbol{ \mathcal A }^+$ such that $\|\mathcal A-\mathcal A_\epsilon\|_F\leq \epsilon$. Let $\{\mathcal X^*,\mathcal Y^*,\Lambda^* \}$ be a solution to \eqref{eq:kkt_nonconvex_reformulation_notau} with respect to $\mathcal A$.
	Let $\mathcal A_\epsilon := \mathcal A + \frac{\epsilon}{2}\mathcal X^*$.  As $\mathcal X^*\in C\cap\mathbb S^{n^d}$,  we have	  $\|\mathcal A-\mathcal A_\epsilon\|_F= \frac{\epsilon}{2}<\epsilon$.  It remains to show that $\mathcal A_\epsilon\in\boldsymbol{ \mathcal A }^+$.

	Denote	 $\Lambda^*_\epsilon:=\Lambda^*+\frac{\epsilon}{2}\mathcal X^*$. Clearly, $\{\mathcal X^*,\mathcal Y^*,\Lambda^*_\epsilon \}$ is a solution to \eqref{eq:kkt_nonconvex_reformulation_notau} with respect to $\mathcal A_\epsilon$; moreover, the definition of $\Lambda^*_\epsilon$ shows that the leading eigenvalue  of $\boldsymbol{ Mat }(\Lambda^*)$ is a simple root. Thus $\mathcal A_\epsilon\in\boldsymbol{ \mathcal A }^+$, as desired.  
\end{proof}
\begin{theorem}\label{prop:open}
	$\boldsymbol{ \mathcal A }^+$ is open in $\mathbb S^{n^d}$.
\end{theorem}

The proof will be presented in the sequel. According to Theorems \ref{prop:dense} and \ref{prop:open},  we have the following conclusions.
\begin{remark}
	\begin{enumerate}
		\item $\boldsymbol{ \mathcal A }^+$ is generic in $\boldsymbol{ \mathcal A }$.
		\item The volume of $\boldsymbol{ \mathcal A }^+$ is positive in $\mathbb S^{n^d}$, namely, that the original problem \eqref{prob:original4} can be solved by the nonconvex ADMM with a theoretical convergence guarantee is not an isolated phenomenon.
		\item It would be more satisfied   if Theorem \ref{prop:dense} is replaced by that $\boldsymbol{ \mathcal A }^+$ is dense in $\mathbb S^{n^d}$; this would be true if $\boldsymbol{ \mathcal A }$ is dense in $\mathbb S^{n^d}$. If this is the case, together with Proposition \ref{prop:open} we can demonstrate that $\boldsymbol{ \mathcal A }^+$ is generic in $\mathbb S^{n^d}$ which confirms the numerical observations. Currently we do not know how to fill this gap, and we leave it as a conjecture. Nevertheless, Theorems \ref{prop:dense} and \ref{prop:open} partly explains   why the nonconvex ADMM is effective in reality.
	\end{enumerate}
\end{remark}

In what follows, we focus on proving Theorem \ref{prop:open}.
First we need the following lemma.
\begin{lemma}\label{prop:diff_value_p}
	Consider the relaxation problem \eqref{prob:relax_convex} rewritten as
	\[	 
	\label{prob:relax_convex_p}
	(R_{\mathcal A})~~\max ~\langle   \mathcal A,\mathcal X\rangle~{\rm s.t.}~       \mathcal X\in C_R\cap    \mathbb S^{n^d}.
	\]
	Denote its optimal value as $V_{R_{\mathcal A}}$. Then it holds that $| V_{R_{\mathcal A}} - V_{R_{\mathcal A+\epsilon\mathcal B}  }  | \leq  \epsilon$ for any $\mathcal A$ and $\mathcal B\in\mathbb S^{n^d}$ with $\|\mathcal B\|_F\leq 1$, where $\epsilon > 0$.
\end{lemma}
\begin{proof}
	It follows   from the definition of $\mathcal B$ and the feasible set of $\mathcal X$ that
	\[V_{R_{\mathcal A+\epsilon \mathcal B}  }  = \max_{\mathcal X\in C_R\cap\mathbb S^{n^d}   }\langle \mathcal A+\epsilon \mathcal B,\mathcal X\rangle \leq \max_{\mathcal X\in C_R\cap\mathbb S^{n^d}   }\langle \mathcal A ,\mathcal X\rangle  + \epsilon \max_{\mathcal X\in C_R\cap\mathbb S^{n^d}   }\langle  \mathcal   B,\mathcal X\rangle  \leq V_{R_{\mathcal A}} + \epsilon. \]
	On the other hand, let $\mathcal X_{\mathcal A}$ be optimal to $(R_{\mathcal A})$. Then 
	\[V_{R_{\mathcal A+\epsilon\mathcal B}  }  \geq   \langle \mathcal A+\epsilon\mathcal B,\mathcal X_{\mathcal A}\rangle  = V_{R_{\mathcal A}} +   \epsilon \langle   \mathcal B,\mathcal X_{\mathcal A}\rangle  \geq V_{R_{\mathcal A}} + \epsilon\min_{\mathcal X\in C_R\cap\mathbb S^{n^d}   }\langle \mathcal B,\mathcal X\rangle \geq V_{R_{\mathcal A}}  - \epsilon, \]
	and the assertion follows.
\end{proof}

As there is no duality gap between \eqref{prob:dual1} and \eqref{prob:relax_convex_y}, it   follows $| V_{D_{\mathcal A}} - V_{D_{\mathcal A+\epsilon\mathcal B}  }  | \leq  \epsilon$, where $V_{D_{\mathcal A+\epsilon\mathcal B}}$ denotes the optimal value of the perturbation problem of \eqref{prob:dual1}:
\begin{equation}\label{prob:dual1_p}
(D_{\mathcal A+\epsilon\mathcal B})~~	\min  ~\sigma_{\max}(\boldsymbol{Mat}(\Lambda))~~{\rm s.t.}~\boldsymbol{Sym}(\Lambda)=\mathcal A + \epsilon\mathcal B,~\boldsymbol{ Mat }(\Lambda)\in\mathbb S^{ n^{d/2}\times n^{d/2} },
\end{equation}
with $\mathcal B\in\mathbb S^{n^d}$, $\epsilon>0$.

 Define
\[
\boldsymbol{ \mathcal L}_{\mathcal A } := \{ \Lambda\mid  \boldsymbol{ Sym }(\Lambda) = \mathcal A ,~\boldsymbol{ Mat }(\Lambda)\in\mathbb S^{n^{d/2}\times n^{d/2}}  \}.
\]
Then the optimal solution set of \eqref{prob:dual1_p} can be written as\footnote{The solution set   may not be a singloten. Consider the very simple example: For any normalized $\mathbf x\in\mathbb R^n$, Let $\mathcal A = \mathbf x^{\circ^d}$. Clearly, $\Lambda^* = \mathcal A$ is a solution to \eqref{prob:dual1}. On the other hand, consider the linear subspace $\{\Lambda\mid \boldsymbol{ Mat }(\Lambda)\in\mathbb S^{n^{d/2}\times n^{d/2}},\boldsymbol{ Sym }(\Lambda)=0,\boldsymbol{ Mat }(\Lambda)\mathbf x^{\otimes^{d/2}} =0 \}$, which is nontrivial. Pick any $\Lambda_0\neq 0$ from the above subspace and let $\Lambda^\prime = \Lambda^* + \alpha\frac{ \Lambda
		_0  }{ \|\Lambda_0\|_F  }$. Then $\Lambda^\prime$ is also a minimizer to \eqref{prob:dual1} when $\alpha \leq 1$. }   
\begin{equation}\label{def:s} \setlength\abovedisplayskip{2pt}
\setlength\abovedisplayshortskip{2pt}
\setlength\belowdisplayskip{2pt}
\setlength\belowdisplayshortskip{2pt}
\boldsymbol{ \mathcal S}_{D_{\mathcal A+\epsilon\mathcal B}}  :=  \boldsymbol{ \mathcal L }_{\mathcal A+\epsilon\mathcal B } \cap \{ \Lambda\mid \boldsymbol{ Mat }(\Lambda) \preceq V_{D_{\mathcal A+\epsilon\mathcal B}} \cdot I  \},
\end{equation} 
where $I\in\mathbb S^{n^{d/2}\times n^{d/2}}$ denotes the identity matrix. In particular, $\boldsymbol{ \mathcal S}_{D_{\mathcal A}}$ is the optimal solution set of the unperturbation problem. Define ${\rm dist}(\mathbf x,X)$ as the distance from a point $\mathbf x$ to a set $X$.

\begin{proof}[Proof of Theorem \ref{prop:open}]
Let $\mathcal A\in\boldsymbol{ \mathcal A }^+$; according to Proposition \ref{prob:sec:hypo:1}, it is equivalent to that the dual \eqref{prob:dual1} with respect to $\mathcal A$ has a solution whose matricization admits a simple leading eigenvalue.  Write the solution as $\Lambda_\mathcal A$ with the   leading eigenvalue of $\boldsymbol{ Mat }(\Lambda_{\mathcal A})$    being simple.  It holds that $\Lambda_{\mathcal A} \in \boldsymbol{ \mathcal S}_{D_{
		\mathcal A}}$.

 We first show that 
 \begin{equation}
\label{eq:sec:hypo:1}
\lim_{\epsilon\rightarrow 0}\nolimits{\rm dist}(\Lambda_{\mathcal A}, \boldsymbol{ \mathcal S}_{D_{\mathcal A+\epsilon\mathcal B}}) =0,~ \forall\mathcal B\in\mathbb S^{n^d} {\rm ~with~} \|\mathcal B\|_F\leq 1,
 \end{equation}
 
 Denote $\mathcal I:= \boldsymbol{ Ten }(I)$, where  $I\in\mathbb R^{n^{d/2}\times n^{d/2}}$ denotes the identity matrix,   and let    $\mathcal I_{\boldsymbol{ Sym }}:= \boldsymbol{ Sym }(\mathcal I)$. By changing the variable as $\Lambda:= V_{ D_{\mathcal A+\epsilon\mathcal B}} \cdot\mathcal I - \Lambda$,  $\boldsymbol{ \mathcal S}_{D_{A+\epsilon\mathcal B}} $ is shifted as
\begin{equation*} \setlength\abovedisplayskip{2pt}
\setlength\abovedisplayshortskip{2pt}
\setlength\belowdisplayskip{2pt}
\setlength\belowdisplayshortskip{2pt}
\boldsymbol{ \mathcal S}_{D_{A+\epsilon\mathcal B}}^\prime: = \boldsymbol{ \mathcal L }_{-\mathcal A-\epsilon\mathcal B +V_{ D_{\mathcal A+\epsilon\mathcal B}} \cdot \mathcal I_{\boldsymbol{ Sym }}}\cap  \{\Lambda\mid\boldsymbol{Mat}(\Lambda)\succeq 0  \}.
\end{equation*}
Accordingly, denote $\Lambda^\prime_{\mathcal A}:= V_{D_{\mathcal A}}\cdot \mathcal I - \Lambda_{\mathcal A}$. Then $\Lambda^\prime_{\mathcal A}\in  \boldsymbol{ \mathcal S}_{D_{A}}^\prime$. By Lemma \ref{prop:diff_value_p} and that $\boldsymbol{ \mathcal S}_{D_{A+\epsilon\mathcal B}}^\prime$ is the intersection of an affine subspace and the positive semidefinite cone, we see that for any $\mathcal B$ with $\|\mathcal B\|_F\leq 1$, it holds that
\begin{equation*}\label{eq:sec:hypothesis:1}
\lim_{\epsilon\rightarrow 0}\nolimits {\rm dist}(\Lambda^\prime_{\mathcal A}, \boldsymbol{ \mathcal S}_{D_{A+\epsilon\mathcal B}}^\prime) = {\rm dist}(\Lambda^\prime_{\mathcal A}, \boldsymbol{ \mathcal S}_{D_{A}}^\prime) = 0,
\end{equation*}
implying that \eqref{eq:sec:hypo:1} holds.

As a result, there is a small enough $\epsilon_0>0$ such that for all $\mathcal B$ with $\|\mathcal B\|_F\leq 1$ and all $\epsilon<\epsilon_0$, there holds 
 $\|  \Lambda_{\mathcal A} - \Lambda_{\mathcal A+\epsilon \mathcal B}\|_F <  \frac{\sigma_{\max}(\boldsymbol{ Mat }(\Lambda_{\mathcal A}))   -\sigma_2( \boldsymbol{ Mat }(\Lambda_{\mathcal A} )  ) }{4}   ,$ 
where $\sigma_2(\cdot)$ is the second largest eigenvalue of the matrix. 
   This combining with Weyl's inequality tells us that 
$| \sigma_i( \boldsymbol{ Mat }(\Lambda_{\mathcal A+\epsilon\mathcal B}) )  -\sigma_i(\boldsymbol{ Mat }(\Lambda_{\mathcal A})) | < \frac{\sigma_{\max}(\boldsymbol{ Mat }(\Lambda_{\mathcal A}))   -\sigma_2( \boldsymbol{ Mat }(\Lambda_{\mathcal A} )  ) }{4}, ~\forall~i,$
   which implies that
\[
\sigma_{\max}(\Lambda_{\mathcal A+\epsilon\mathcal B}) > \frac{ 3\sigma_{\max}(\Lambda_{\mathcal A})   +\sigma_2(\Lambda_{\mathcal A}  )    }{4} > \frac{  \sigma_{\max}(\Lambda_{\mathcal A})   +3\sigma_2(\Lambda_{\mathcal A}  )    }{4}>\sigma_2(\Lambda_{\mathcal A+\epsilon\mathcal B}), 
\]
i.e., $\mathcal A+\epsilon\mathcal B\in\boldsymbol{ \mathcal A }^+$.  As a consequence, the open ball $\{\mathcal C\in\mathbb S^{n^d}\mid \|\mathcal A-\mathcal C\|_F < \epsilon_0  \} \subset \boldsymbol{ \mathcal A }^+$, and hence $\boldsymbol{ \mathcal A }^+$ is open in $\mathbb S^{n^d}$. The proof has been completed.
\end{proof}

\subsection{Practical considerations: feasible and $O(\epsilon)$-optimal solution}
Let $\{\mathcal X^k,\mathcal Y^k,\Lambda^k \}$ be generated by the algorithm. According to Theorem \ref{prop:equivalance},  one obtains a rank-1 tensor provided that $\mathcal X^k=\mathcal Y^k$. However, in reality, due to the rounding errors, or the stopping criterion, there may exist a gap between $\mathcal X^k$ and $\mathcal Y^k$, which might result in that neither $\mathcal X^k$ nor $\mathcal Y^k$ is a rank-1 tensor.  Using perturbation analysis, we show that  a normalized rank-1 solution that is close to $\mathcal X^k$ or $\mathcal Y^k$ can be computed in polynomial time. Moreover,  under checkable conditions, such feasible solution is
$O(\epsilon)$-optimal    to the original problem \eqref{prob:original4}. Here we do not assume any prior information between $\{\mathcal X^k,\mathcal Y^k,\Lambda^k \}$ and the optimal solution $\{\mathcal X^*,\mathcal Y^*,\Lambda^* \}$.

Suppose in practice, the algorithm stops at the $k$-th iteration, and the  following has been observed:
\begin{assumption}\label{ass:sec:eps}
	\begin{enumerate}
		\item $\{ \mathcal X^{k},\mathcal Y^{k}, \Lambda^{k}  \}$ is feasible to \eqref{prob:relax3};
		\item $ \|  \{ \mathcal X^{k },\mathcal Y^{k }, \Lambda^{k }  \}  - \{ \mathcal X^{k-1},\mathcal Y^{k-1}, \Lambda^{k-1}  \}\|_F= \epsilon    $, with $\epsilon>0$ sufficiently small.
	\end{enumerate}
\end{assumption} 
Clearly,  the above assumptions are quite natural. In addition, we require that
\begin{assumption}\label{ass:par_sym}
	$\mathcal X^k$ is partially symmetric.
\end{assumption}

Here partial symmetry is defined in $\mathbb R^{n^d}$ as follows.   If for   even   $d$, 
$\mathcal A_{i_1\cdots i_d} = \mathcal A_{ \pi_1 \pi_2} = \mathcal A_{i_{d/2+1}\cdots i_d,i_1\cdots i_{d/2}}$ for any $\pi_1\in \pi(i_1\cdots i_{  d/2 })$ and any $\pi_2\in\pi(i_{  d/2 +1}\cdots i_d)$, and if  for odd order $d$, $\mathcal A_{i_1\cdots i_d} = \mathcal A_{ \pi_1 \pi_2 }  $ for any $\pi_{1}\in \pi(i_1\cdots i_{\lfloor d/2\rfloor})$ and any $\pi_{2}\in\pi(i_{\lfloor d/2\rfloor+1}\cdots i_d)$, then we call $\mathcal A$ partially symmetric. Partially symmetric tensors are denoted as $\mathbb S^{n^d}_P$.
\begin{remark}
	The partial symmetry of $\mathcal X^k$ is easily preserved during the iterates once $\Lambda^0\in\mathbb S_P^{n^d}$ (in particular, $\Lambda^0=\mathcal A$). 
	We only discuss the even order case, while it  is analogous when $d$  is odd. Note that if $\{\Lambda^0,\mathcal Y^0\}\in\mathbb S_P^{n^d}\times \mathbb S^{n^d}$, then
	$-\Lambda^{0} - \tau\mathcal Y^{0} \in\mathbb S_P^{n^d}$, and the  smallest eigenvectors of $\boldsymbol{ Mat }(-\Lambda^{0} - \tau\mathcal Y^{0})$, denoted as $\mathbf x\in\mathbb R^{n^{d/2}}$, satisfies that its tensorization is symmetric\footnote{Here the tensorization of $\mathbf x\in\mathbb R^{n^{d/2}}$ means the tensor $\textsf{ reshape}(\mathbf x,\overbrace{n,\ldots,n}^{d/2})\in\mathbb S^{n^{d/2}}$.}. Since $\boldsymbol{ Mat }(\mathcal X^1) =\mathbf x\mathbf x^\top$, 
	it then holds that $\mathcal X^1\in\mathbb S^{n^{d/2}}_P$, and so $\Lambda^1 = \Lambda^0 - \tau(\mathcal X^1-\mathcal Y^1)\in\mathbb S_P^{n^{d/2}}$.
	Inductively, $\mathcal X^k\in\mathbb S_P^{n^{d}},~\forall k$.
\end{remark}

With the above practical assumptions, the following, which can be seen as a perturbation version of Theorem \ref{prop:equivalance}, is the first main result.
\begin{theorem}\label{th:sec:eps:epsilon_solution}
	Let $d$ be fixed. Let $ \{ \mathcal X^{k},\mathcal Y^{k}, \Lambda^{k}  \}$ satisfy Assumptions \ref{ass:sec:eps} and \ref{ass:par_sym}. Then in polynomial time, one can find a symmetric normalized rank-1 tensor $\overline{\mathcal X}$, namely, $\overline{\mathcal X}$ is feasible to \eqref{prob:original4}, such that $\|\mathcal X^k-\overline{\mathcal X}\|_F = O(\epsilon)$.
\end{theorem} 

The proof is left to the supplemental materials.  With the above $\overline{\mathcal X}$ at hand, we then have:
\begin{theorem}\label{th:sec:eps:epsilon_optimal_solution}
	Let $\overline{\mathcal X}$ be given as above. If   $\mathcal X^{k }\in\arg\min_{\mathcal X\in C}\langle -\Lambda^{k-1}-\tau\mathcal Y^{k-1}+\tau\mathcal X^{k },\mathcal X\rangle $, then there holds $V_O\leq \langle-\mathcal A,\overline{\mathcal X}\rangle\leq V_O+O(\epsilon)$, where $V_O$ denotes the optimal value of \eqref{prob:original4}.
\end{theorem}
\begin{proof}
	Similar to the proof of Theorem \ref{th:anytau}, the condition means that $\sigma_{n^{d/2}}(-\boldsymbol{ Mat }(\Lambda^{k-1} +\tau\mathcal Y^{k-1}))+\tau$ is the smallest eigenvalue of $-\boldsymbol{ Mat }(\Lambda^{k-1}+\tau\mathcal Y^{k-1}-\tau\mathcal X^{k })$. Let $\mathcal X^*$ be optimal to \eqref{prob:original4}. It   follows from Theorem \ref{th:sec:eps:epsilon_solution} and the conditions that
	\begin{eqnarray*}\setlength\abovedisplayskip{1pt}
		\setlength\abovedisplayshortskip{1pt}
		\setlength\belowdisplayskip{1pt}
		\setlength\belowdisplayshortskip{1pt}
		\langle -\mathcal A,\overline{\mathcal X}\rangle &=& \langle -\Lambda^{k-1},\overline{\mathcal X}\rangle \\
		&=&\langle -\Lambda^{k-1}-\tau\mathcal Y^{k-1}+\tau\mathcal X^{k},\overline{\mathcal X} \rangle+O(\epsilon)\\
		&=& \langle -\Lambda^{k-1}-\tau\mathcal Y^{k-1}+\tau\mathcal X^{k },\mathcal X^{k } \rangle + O(\epsilon)\\
		&\leq&  \langle -\Lambda^{k-1}-\tau\mathcal Y^{k-1}+\tau\mathcal X^{k },\mathcal X^* \rangle + O(\epsilon)\\
		&=&\langle -\Lambda^{k-1},\mathcal X^*\rangle +O(\epsilon)\\
		&=&\langle -\mathcal A,\mathcal X^*\rangle +O(\epsilon) = V_O+O(\epsilon),
	\end{eqnarray*}
	where the first and the fifth equalities are due to Proposition \ref{prop:sym_invariance} and that $\boldsymbol{ Sym }(\Lambda^{k-1})=\mathcal A$ (Remark \ref{rmk:sec:convergence:1}), the second and the fourth equalities come from Assumption \ref{ass:sec:eps}, and the third one follows from Theorem \ref{th:sec:eps:epsilon_solution}; the inequality is due to the optimality of $\mathcal X^{k}$.   On the other hand, it follows from the feasibility of $\overline{\mathcal X}$  to \eqref{prob:original4} that $V_O\leq \langle-\mathcal A,\overline{\mathcal X}\rangle$. The results follow.
\end{proof}

\section{Numerical Experiments}\label{sec:numer}

 All the   computations are conducted on an Intel i7-7770 CPU desktop computer with 32 GB of RAM. The supporting software is Matlab R2015b.  
 Our   code is available online for public use\footnote{\url{https://drive.google.com/drive/folders/1wTrDs-TQfvDERCQLSvP9PtDD8ZUGlDYl}}.
 
 \paragraph{Settings} Unless otherwise specified, the initial guess is $\{\mathcal X^0,\mathcal Y^0,\Lambda^0  \}= \{\boldsymbol{ 0 },\boldsymbol{ 0 },\mathcal A   \}$ where $\boldsymbol{ 0 }$ denote the tensor with each entry being zero; the stopping criterion is 
 $$\max\{ \|\mathcal X^k-\mathcal Y^k\|_F,\|\mathcal X^{k+1}-\mathcal X^k\|_F,\|\mathcal Y^{k+1}-\mathcal Y^k\|_F/\|\mathcal Y^k\|_F  \} \leq \epsilon$$
where $\epsilon=10^{-4}$ or $k\geq 1000$. 
 After $\mathcal X$ is computed, the procedure in the proof of Theorem \ref{th:sec:eps:epsilon_solution} can be used to get the associated eigenvector $\mathbf x$ in polynomial-time.  $\tau$ in the augmented Lagrangian function is the only parameter in the algorithm. To select $\tau$, except those small examples in Section \ref{sec:small}, we first normalize the data tensor $\mathcal A$ such that $\mathcal A = \mathcal A/\|\mathcal A\|_F$, and set $\tau=0.1$ when $d$ is even, and $\tau=0.5$ when $d$ is odd empirically. To compute the $\mathcal X$-subprolems, namely, to compute the leading eigenvalue/singular value of a matrix, we respectively employ the Matlab built-in function \textsf{eigs}, and the function \textsf{lansvd} available in the Matlab package PROPACK\footnote{\url{http://sun.stanford.edu/~rmunk/PROPACK/}}, which are found to be  relatively more efficient and stable, among   others.

\subsection{Small examples} \label{sec:small}

 \begin{example}(\cite[Example 3.6]{kolda2011shifted}, \cite[Example 3.2]{nie2014semidefinite}) Consider
	$\mathcal A\in\mathbb S^{3^3}$ which is given by
	\begin{eqnarray*}
&&\mathcal A_{111} = -0.1281,\mathcal A_{112} = 0.0516,\mathcal A_{113} = -0.0954,\mathcal A_{122} = -0.1958, \mathcal A_{123} = -0.1790,	\\
&&\mathcal A_{133} = -0.2676,\mathcal A_{222} = 0.3251,\mathcal A_{223} = 0.2513,\mathcal A_{233} = 0.1773,\mathcal A_{333} = 0.0338. 
	\end{eqnarray*}
\end{example}
Setting $\epsilon=10^{-5}$ and $\tau=1$, the nonconvex ADMM successfully finds the leading eigenpair $(\sigma_{\max},\mathbf x) = (0.8730, (-0.3921,	 0.7248,	 0.5664))$ of $\mathcal A$ in $0.1355$ seconds using $13$ iterates. $\{\mathcal X^*,\mathcal Y^*,\Lambda^* \}$ returned by the algorithm meets the optimality condition \eqref{eq:kkt_nonconvex_reformulation_notau}, which means that the algorithm automatically  identifies that $0.8730$ is the leading eigenvalue. We have tried other $\tau$. For any $\tau\in \{10^{-3},10^{-2},10^{-1},1,10,10^{2},10^{3}  \}$, we also observe that the algorithm all finds the leading eigenvalue.

\begin{example}(\cite[Example 3.3]{nie2014semidefinite})
	Consider	$\mathcal A\in\mathbb S^{3^3}$ which is given by
\begin{eqnarray*}
&&\mathcal A_{111} = 0.0517,\mathcal A_{112} = 0.3579,\mathcal A_{113} = 0.5298,\mathcal A_{122} = 0.7544,\mathcal A_{123} = 0.2156,\\
&&\mathcal A_{133} = 0.3612,\mathcal A_{222} = 0.3943,\mathcal A_{223} = 0.0146,\mathcal A_{233} = 0.6718,\mathcal A_{333} = 0.9723.
\end{eqnarray*}
\end{example}
Setting $\epsilon=10^{-5}$ and $\tau=1$, the nonconvex ADMM successfully finds the leading eigenpair $(\sigma_{\max},\mathbf x) = (2.1110, (0.5204,	0.5113,	0.6839))$ of $\mathcal A$ in $0.1265$ seconds using $37$ iterates.  We also observe that $\{\mathcal X^*,\mathcal Y^*,\Lambda^* \}$ returned by the algorithm meets \eqref{eq:kkt_nonconvex_reformulation_notau}, i.e., by checking \eqref{eq:kkt_nonconvex_reformulation_notau},  the algorithm automatically  identifies that $\sigma_{\max}$ is the leading eigenvalue. For any $\tau\in \{10^{-3},10^{-2},10^{-1},1,10,10^{2},10^{3}  \}$, we also observe that the  algorithm all finds the leading eigenvalue. 

\begin{example}(\cite[Example 3.1]{jiang2015tensor})
Consider				$\mathcal A\in\mathbb S^{3^4}$ which is given by
	\begin{eqnarray*}
&&\mathcal A_{1111} = 0.2883, \mathcal A_{1112} = -0.0031,\mathcal A_{1113} = 0.1973, \mathcal A_{1122} = -0.2485,\mathcal A_{1123} = -0.2939,\\
&&\mathcal A_{1133} = 0.3847, \mathcal A_{1222} = 0.2972, \mathcal A_{1223} = 0.1862, \mathcal A_{1233} = 0.0919,\mathcal A_{1333} = -0.3619,\\
&&\mathcal A_{2222} = 0.1241, \mathcal A_{2223} = -0.3420, \mathcal A_{2233} = 0.2127, \mathcal A_{2333} = 0.2727, \mathcal A_{3333} = -0.3054.
	\end{eqnarray*}
\end{example}
Setting $\epsilon=10^{-5}$ and $\tau=0.1$, the nonconvex ADMM successfully finds the leading eigenpair $(\sigma_{\max},\mathbf x) = (0.8893, (-0.6672,	-0.2471,	0.7027))$ of $\mathcal A$ in $0.1633$ seconds using $64$ iterates. We also observe that $\{\mathcal X^*,\mathcal Y^*,\Lambda^* \}$ returned by the algorithm meets \eqref{eq:kkt_nonconvex_reformulation_notau}. For $\tau\in \{10^{-3},10^{-2},10^{-1} \}$, the algorithm finds the leading eigenvalue; for $\tau\in \{1,10,10^2,10^3  \}$, $0.8169$ is returned, which is still an eigenvalue but not the leading one.

Setting $\tau=0.1$, the algorithm can also successfully find the leading eigenpair $(1.0954, (0.5915,   -0.7467,   -0.3043))$ of $ -\mathcal A $ in $0.2738$ seconds using $182$ iterates; $\{\mathcal X^*,\mathcal Y^*,\Lambda^* \}$ returned by the algorithm meets \eqref{eq:kkt_nonconvex_reformulation_notau}. For any $\tau\in \{10^{-3},10^{-2},10^{-1},1$, $10,10^{2},10^{3}  \}$, we also observe that the algorithm all finds the leading eigenvalue.

\begin{example}(\cite[Example 3.2]{jiang2015tensor})\label{ex:6.4}
			$\mathcal A\in\mathbb S^{3^4}$ which is given by
			\begin{small}
\begin{eqnarray*}
&&\mathcal A_{1111} = 0.74694, \mathcal A_{1112} =-0.435103, \mathcal A_{1122} = 0.454945, \mathcal A_{1222}=0.0657818,\mathcal A_{2222}=1,\\
&&\mathcal A_{1113}=0.37089, \mathcal A_{1123} = -0.29883, \mathcal A_{1223} = -0.795157, \mathcal A_{2223}=0.139751, \mathcal A_{1133}=1.24733,\\
&&\mathcal A_{1233}=0.714359, \mathcal A_{2233}=0.316264, \mathcal A_{1333}=-0.397391, \mathcal A_{2223}=-0.405544, \mathcal A_{3333}= 0.794869, \\
&&\mathcal A = \boldsymbol{ Sym }(\mathcal A).
\end{eqnarray*}
\end{small}
\end{example}
Setting $\epsilon=10^{-5}$ and $\tau=0.1$, the nonconvex ADMM successfully finds the leading eigenpair $(\sigma_{\max},\mathbf x) = (1.0031, (-0.0116,   -0.9992,   -0.0382))$ of $\mathcal A$ in $0.1961$ seconds using $68$ iterates. $\{\mathcal X^*,\mathcal Y^*,\Lambda^* \}$ returned by the algorithm meets \eqref{eq:kkt_nonconvex_reformulation_notau}. 
 For any $\tau\in \{10^{-3},10^{-2},10^{-1},1,10,10^{2},10^{3}  \}$, we also observe that the algorithm all finds the leading eigenvalue.

Setting $\tau=0.1$, the algorithm can also   find  the leading eigenpair \begin{small}$(\sigma_{\max},\mathbf x) = (-0.3837, (-0.4360 ,  -0.5954 ,  -0.6748))$\end{small} of   $ -\mathcal A$ in $0.1485$ seconds using $67$ iterates. Although $-0.3837$ is the leading eigenvalue which is verified by the SDP relaxation, $\{\mathcal X^*,\mathcal Y^*,\Lambda^* \}$ meets \eqref{eq:kkt_nonconvex_reformulation} instead of \eqref{eq:kkt_nonconvex_reformulation_notau}. This confirms iterm 2 of the discussions right after Theorem \ref{th:convergence_admm_even}.
For other $\tau$, the results are listed in Table \ref{tab:effect_tau}.

\begin{example}(\cite[Example 3.8]{nie2014semidefinite})
	$\mathcal A\in\mathbb S^{3^6}$ which is associated with the following polynomial
	$$ \setlength\abovedisplayskip{2pt}
	\setlength\abovedisplayshortskip{2pt}
	\setlength\belowdisplayskip{2pt}
	\setlength\belowdisplayshortskip{2pt} f(\mathbf x) = 2\|\mathbf x\|^6 - ( \mathbf x_1^4\mathbf x_2^2 +\mathbf x_1^2\mathbf x_2^4 +\mathbf x_3^6 - 3\mathbf x_1^2\mathbf x_2^2\mathbf x_3^2 ).$$
\end{example}
Normalizing such that $\|\mathcal A\|_F=1$, setting $\epsilon=10^{-5}$ and $\tau=0.1$, the nonconvex ADMM successfully finds the leading eigenpair $(\sigma_{\max},\mathbf x) = (2, (0,   1,   0))$ of $\mathcal A$ in $0.0548$ seconds using $29$ iterates. If setting $\tau=0.5$, then it will find $(2, (0,   1,   0))$ which is still global.  For any $\tau\in \{10^{-2},10^{-1},1,10,10^{2},10^{3}  \}$, we also observe that the algorithm all finds the leading eigenvalue. However, the algorithm does not converge when $\tau=10^{-3}$. This may because there does not exist $\Lambda^*$ satisfying the hypothesis of Theorem \ref{th:convergence_admm_even} (using SDP relaxation  starting from different initial points, we always observe that the leading eigenvalue of the resulting $\boldsymbol{ Mat }(\Lambda^*)$ is not simple).

For the same $\tau$, the algorithm can also   find  the leading eigenpair $(\sigma_{\max},\mathbf x) = (-1, (0 ,  0 ,  1))$ of   $ -\mathcal A$ in $0.2568$ seconds using $47$ iterates. For any $\tau\in \{10^{-3},10^{-2},10^{-1}$, $1,10,10^{2},10^{3}  \}$, we also observe that the algorithm all finds the leading eigenvalue.
\paragraph{Influence  of $\tau$}
  We illustrate the results   with $\tau$ varying from $10^{-3}$ to $10^3$ in details, where $\epsilon=10^{-8}$. The   tensor is $-\mathcal A$ where $\mathcal A$ is defined in Example \ref{ex:6.4}, with results shown in Table \ref{tab:effect_tau}. Varying from $10^{-3}$ to $0.5$, the algorithm finds the leading eigenvalue of $-\mathcal A$. Varying from $1$ to $10^3$, the algorithm gets $-0.3904$ which is not the leading one. Besides $-0.3837$ and $-0.3904$, no other eigenvalues have been found by the algorithm no matter what $\tau$ is chosen. This shows that even if the nonconvex ADMM cannot find the global solution, it can still find a high-quality one. We can also observe that the algorithm is more efficient with a reasonable $\tau$.
 
\begin{table}[htbp] \setlength\abovedisplayskip{2pt}
	\setlength\abovedisplayshortskip{2pt}
	\setlength\belowdisplayskip{2pt}
	\setlength\belowdisplayshortskip{2pt}
	\centering
	\caption{Eigenvalue of $-\mathcal A$ where $\mathcal A$ is defined in Example \ref{ex:6.4} with different $\tau$}
	\begin{mytabular}{ccccccccc}
		\toprule
		$\tau$ & $10^{-3}$ & $10^{-2}$ & $10^{-1}$ & $0.5$  & $1$ & $10$ & $10^2$ & $10^3$ \\
		\midrule
		$V$     & -0.3837  & -0.3837  & -0.3837  & -0.3837  & -0.3904  & -0.3904  & -0.3904  & -0.3904  \\
		Iter. & 5502  & 550   & 96    & 99    & 46    & 147    & 1033   & 6940  \\
		\bottomrule
	\end{mytabular}%
	\label{tab:effect_tau}%
\end{table}%

\paragraph{Summary} For all the examples, we find that  the algorithm: 1) converges with most   $\tau>0$; 2) converges even if the hypothesis on $\Lambda^*$ cannot be met;   3) converges to the global solution efficiently, when $\tau$ lies in a certain range; 4) satisfies \eqref{eq:kkt_nonconvex_reformulation_notau}, namely, $\sigma_{\max}$ of $\mathcal A$ is also the leading eigenvalue of $\boldsymbol{ Mat }(\Lambda^*)$, when $\tau$ lies in a certain range; 5) if $\tau$ is chosen larger, then the algorithm might converge to other eigenvalues, but the solution quality is still good.

\subsection{Structured tensors}
Four classes of large-scale structured tensors are considered in this section. The first class is the Hilbert tensors  \cite{song2014infinite}, which is a generalization of the Hilbert matrix. The other three classes follow those of \cite{nie2014semidefinite}.

\paragraph{Hilbert tensors} The Hilbert tensor  $\mathcal A\in\mathbb S^{n^d}$   is defined by   \cite{song2014infinite}
\begin{equation}\label{tensor:hilbert}\small  \setlength\abovedisplayskip{2pt}
\setlength\abovedisplayshortskip{2pt}
\setlength\belowdisplayskip{2pt}
\setlength\belowdisplayshortskip{2pt}
\mathcal A(i_1,\ldots,i_d) = \frac{1}{i_1+\cdots + i_d-d+1}.
\end{equation}
\paragraph{Structured tensors defined by  logarithm functions} (\cite[Example 3.7]{nie2014semidefinite})     
\begin{equation}\label{tensor:log}\small  \setlength\abovedisplayskip{2pt}
\setlength\abovedisplayshortskip{2pt}
\setlength\belowdisplayskip{2pt}
\setlength\belowdisplayshortskip{2pt}
\mathcal A(i_1,\ldots,i_d) = (-1)^{i_1}\ln (i_1) + \cdots + (-1)^{i_d}\ln(i_d),
\end{equation}
\paragraph{Structured tensors defined by arctangent functions} (\cite[Example 3.6]{nie2014semidefinite}) 
\begin{equation}\label{tensor:atan}\small  \setlength\abovedisplayskip{2pt}
\setlength\abovedisplayshortskip{2pt}
\setlength\belowdisplayskip{2pt}
\setlength\belowdisplayshortskip{2pt}
\mathcal A(i_1,\ldots,i_d) =    \arctan \left(  (-1)^{i_1} \frac{i_1}{n} \right) + \cdots +  \arctan \left(  (-1)^{i_d} \frac{i_d}{n} \right),
\end{equation}
\paragraph{Structured tensors defined by fraction functions} (\cite[Example 3.5]{nie2014semidefinite}) 
\begin{equation}\label{tensor:fraction}\small  \setlength\abovedisplayskip{2pt}
\setlength\abovedisplayshortskip{2pt}
\setlength\belowdisplayskip{2pt}
\setlength\belowdisplayshortskip{2pt}
\mathcal A(i_1,\ldots,i_d) = \frac{(-1)^{i_1}}{i_1} + \cdots + \frac{(-1)^{i_d}}{i_d}.
\end{equation}

\begin{table*}[h] \footnotesize
		\renewcommand\arraystretch{0.7}
		
	\begin{floatrow}
		\capbtabbox{
			
			\setlength{\tabcolsep}{0.3mm}			
			
			\begin{tabular}{cc|cc|cccc}
				\toprule
				&       & \multicolumn{2}{c}{SDP \cite{nie2014semidefinite}} & \multicolumn{4}{c}{Algorithm \ref{alg:admm}} \\
				\cmidrule{3-4}  \cmidrule{5-8}  $d$     & $n$     & Time  & $V$     & Time  & Iter. & $V$     & Opt? \\
				\midrule
				& 20    & 2.12  & 4.18  & 0.15  & 24    & 4.18  & Y \\
				& 40    & 45.17  & 5.81  & 0.20  & 24    & 5.81  & Y \\
				& 60    & 565.34  & 7.06  & 0.56  & 24    & 7.06  & Y \\
				3     & 80    & 4503.37  & 8.12  & 0.84  & 24    & 8.12  & Y \\
				& 200   & -    & -    & 8.90  & 24    & 12.76  & Y \\
				& 300   & -    & -    & 47.51  & 24    & 15.60  & Y \\
				& 400   & -    & -    & 73.93  & 24    & 18.00  & Y \\
				& 500   &   -   &    - &   96.87    &    24   &    20.11   & Y \\
				\midrule
					&10 & 1.60 &	6.53 &	0.26 &	82 &	6.53 &	Y\\	
				& 20    & 11.87  & 12.51  & 5.28  & 71    & 12.51  & Y \\
				& 40    & 176.73  & 24.50  & 35.20  & 69    & 24.50  & Y \\
				    4 & 60    & 1000.18  & 36.50  & 118.75  & 66    & 36.50  & Y \\
				& 80    & 6205.97  & 48.50  & 264.83  & 66    & 48.50  & Y \\
			 	&90   & -    &-    & 304.55 & 63    & 54.50  & Y \\
				& 100   & -    & -    & 396.50  & 65    & 60.50  & Y \\
				 	& 120   & -     & -     & 607.93 & 63    & 72.50  & Y \\
				\midrule
				& 5&	1.23 &	6.11 &	0.22 &	24 &	6.11 &	Y\\				
			5	& 10    & 24.47  & 15.75  & 1.35  & 23    & 15.75  & Y \\
				     & 20    & 925.58  & 42.61  & 10.01  & 23    & 42.61  & Y \\
				& 30    & -     & -     & 42.43  & 23    & 77.16  & Y \\
				\midrule
				& 5     & 1.14  & 11.14  & 0.83  & 77    & 11.14  & Y \\
				6     & 10    & 29.21  & 40.43  & 8.43  & 71    & 40.43  & Y \\
				& 15    & 183.10  & 88.23  & 80.14  & 69    & 88.23  & Y \\
				& 20    & 771.84  & 154.53  & 306.71  & 68    & 154.53  & Y \\
				\bottomrule
			\end{tabular}%

		}{
			\caption{Hilbert tensors \eqref{tensor:hilbert}.}
			\label{tab:hilbert}
		}
		\capbtabbox{
			
			\setlength{\tabcolsep}{0.3mm}

			\begin{tabular}{cc|cc|cccc}
				\toprule
				&       & \multicolumn{2}{c}{SDP \cite{nie2014semidefinite}} & \multicolumn{4}{c}{Algorithm \ref{alg:admm}} \\
				\cmidrule{3-4}\cmidrule{5-8}  $d$&  $n$     & Time  & $V$     & Time  & Iter. & $V$     & Opt? \\
				\midrule
				&    20    & 2.55  & 246.19  & 0.17  & 22    & 246.19  & Y \\
				&  40    & 72.58  & 865.65  & 0.32  & 22    & 865.65  & Y \\
				&60    & 555.58  & 1782.79  & 1.11  & 23    & 1782.79  & Y \\
				3&    80    & 4053.95  & 2960.46  & 2.03  & 23    & 2960.46  & Y \\
				&  200   & -     & -     & 15.09  & 24    & 14501.86  & U \\
				&300   & -     & -    & 65.42  & 25    & 28970.18  & U \\
				&400   & -    & -    & 90.30  & 26    & 47166.83  & U \\
				&500   & -    & -    & 205.08  & 27    & 68710.81  & U \\
				\midrule	
					&	10	&1.21 &	248.30 &	0.29 &	90 &	248.30 &	Y\\		
				&20    & 2.94  & 1253.38  & 4.00  & 88    & 1253.38  & Y \\
				&40    & 56.04  & 6193.99  & 22.41  & 86    & 6193.99  & Y \\
				4&60    & 519.35  & 15592.50  & 90.37  & 86    & 15592.50  & Y \\
				&80    & 2970.55  & 29869.49  & 247.34  & 85    & 29869.49  & Y \\
				& 90 & - & - & 340.92 & 	85  & 	38932.76	&Y\\
				&100   & -    & -    & 451.32  & 85    & 49320.50  & Y \\
				& 120 & -& - & 705.07 &	83 & 	74174.21 &	Y\\
				\midrule
				&5&	0.63 &	110.01 &	0.20& 	26& 	110.01 &	Y\\				
			5	&10    & 4.60  & 883.28  & 4.55  & 25    & 883.28  & Y \\
				&20    & 912.35  & 6236.72  & 26.06  & 25    & 6236.72  & Y \\
				&30    & -    & -     & 134.09  & 25    & 19438.60  & U \\
				\midrule
				&5     & 2.14  & 164.46  & 0.92  & 64    & 164.46  & Y \\
				6&10    & 29.89  & 3086.58  & 5.74  & 80    & 3086.58  & Y \\
				&15    & 192.46  & 9642.00  & 55.99  & 71    & 9642.00  & Y \\
				&20    & 1019.70  & 30528.86  & 394.02  & 77    & 30528.86  & Y \\
				\bottomrule
			\end{tabular}%

		}{
			\caption{Structured tensors with log. functions \eqref{tensor:log}.}
			\label{tab:log}
		}
	\end{floatrow}
\end{table*}

\paragraph{Settings} The order $d$ varies from $3$ to $6$. Depending on $d$,  the dimension varies from $5$ to $500$. The method based on SDP relaxation is used as a baseline for comparisons. In particular, we employ that of Nie and Wang \cite{nie2014semidefinite} as the baseline, because their method is based on SDPNAL \cite{zhao2010newton}, which is  usually faster than the implement of \cite{jiang2015tensor}. {Here we remark that, as the method of \cite{nie2014semidefinite} is proposed to find a best rank-1 approximation, when $d$ is even, it will solve both  $\max_{\|\mathbf x\|=1} \langle \mathcal A,\mathbf x^{\circ^d}\rangle$ and $\min_{\|\mathbf x\|=1}\langle \mathcal A,\mathbf x^{\circ^d}\rangle$. In view of this, to give a fair comparison, when $d$ is even, we have modified the code of \cite{nie2014semidefinite} so that it only solves $\max_{\|\mathbf x\|=1} \langle \mathcal A,\mathbf x^{\circ^d}\rangle$.  } In the cases that $d=3$ and $n\geq 200$, $d=4$ and $n=100$, $d=5$ and $n=30$, SDP is too time-consuming, and we do not run it on the corresponding instances, where we mark the results as `-'.
For our algorithm, 
all the settings follow those introduced in the beginning of this section. 

\paragraph{Remarks on the tables} The results are illustrated in Tables \ref{tab:hilbert}, \ref{tab:log}, \ref{tab:atan} and \ref{tab:fraction}, respectively. All the tables have the same format. The first two columns stand for order and dimension, respectively; the next two columns refer to the CPU time and the objective value returned by the SDP. The unit of time is second. The five to the seven columns stand for the CPU time, iterates and the objective value returned by Algorithm \ref{alg:admm}. The last column denotes whether our approach can find the optimal solution. Here by optimal solution, we mean that, if it is equal to that found by SDP, or it satisfies the optimality condition \eqref{eq:kkt_nonconvex_reformulation}, then we mark it as `Y'; if it is  not equal to the one found by SDP, then we mark it as `N'. Another case is that when SDP is not available, and \eqref{eq:kkt_nonconvex_reformulation} is not met, which means that we are unclear whether the solution is optimal or not. In this case, we mark it as `U(nclear)'.

\begin{table*}\footnotesize
			\renewcommand\arraystretch{0.7}
	\begin{floatrow}
		\capbtabbox{

			\setlength{\tabcolsep}{0.3mm}			
			\begin{tabular}{cc|cc|cccc}
				\toprule
				&       & \multicolumn{2}{c}{SDP \cite{nie2014semidefinite}} & \multicolumn{4}{c}{Algorithm \ref{alg:admm}} \\
				\cmidrule{3-8}    $d$     & $n$     & Time  & $V$     & Time  & Iter. & $V$     & Opt? \\
				\midrule
				& 20    & 2.00  & 55.59  & 0.16  & 22    & 55.59  & Y \\
				& 40    & 43.83  & 150.98  & 0.22  & 22    & 150.98  & Y \\
				& 60    & 319.15  & 273.52  & 0.35  & 22    & 273.52  & Y \\
				3     & 80    & 3308.25  & 418.14  & 0.74  & 23    & 418.14  & Y \\
				& 200   &-    &-    & 15.65  & 24    & 1631.65  & U \\
				& 300   & -     & -    & 51.35  & 25    & 2988.88  & U \\
				& 400   & -     & -    & 113.03  & 26    & 4595.01  & U \\
				& 500   & -     & -    & 229.93  & 26    & 6416.13  & U \\
				\midrule	
				&10	&0.78 &	77.07 &	0.28 &	90 &	77.07 &	Y\\					
				& 20    & 2.56  & 282.97  & 1.20  & 88    & 282.97  & Y \\
					& 40    & 68.24  & 1080.77  & 13.94  & 86    & 1080.77  & Y \\
				4& 60    & 883.72  & 2393.25  & 60.46  & 86    & 2393.25  & Y \\
				& 80    & 8738.37  & 4220.42  & 254.23  & 86    & 4220.42  & Y \\
				& 90 & - & -& 390.39	& 85 &	5327.02 & Y\\
				& 100   & -    & -    & 467.00  & 85    & 6562.28  & Y \\
				& 120 & -& - & 767.75 &	86 &	9418.83 & Y\\
				\midrule
				&5&	1.25 &	60.37 &	0.32 &	27 &	60.37 &	Y	\\			
				5	& 10    & 9.17  & 273.40  & 0.93  & 25    & 273.40  & Y \\
				& 20    & 556.23  & 1407.83  & 7.39  & 25    & 1407.83  & Y \\
				& 30    & -     & -     & 59.73  & 25    & 3751.26  & U \\
				\midrule
				& 5     & 0.71  & 70.66  & 0.70  & 60    & 70.66  & Y \\
				6     & 10    & 64.98  & 953.06  & 10.44  & 79    & 953.06  & Y \\
				& 15    & 296.97  & 2385.36  & 80.52  & 69    & 2385.36  & Y \\
				& 20    & 1364.63  & 6890.50  & 395.45  & 77    & 6890.50  & Y \\
				\bottomrule
			\end{tabular}%

		}{
			\caption{Structured tensors with arctan. functions \eqref{tensor:atan}.}
			\label{tab:atan}
		}
		\capbtabbox{
			
			\setlength{\tabcolsep}{0.3mm}

			\begin{tabular}{cc|cc|cccc}
				\toprule
				&	& \multicolumn{2}{c}{SDP \cite{nie2014semidefinite}} & \multicolumn{4}{c}{Algorithm \ref{alg:admm}} \\
				\cmidrule{3-8} $d$&   $n$     & Time  & $V$     & Time  & Iter. & $V$     & Opt? \\
				\midrule
         & 20    & 2.55  & 34.16  & 0.24  & 24    & 34.16  & Y \\
          & 40    & 45.20  & 65.93  & 0.23  & 24    & 65.93  & Y \\
          & 60    & 625.85  & 97.18  & 1.22  & 25    & 97.18  & Y \\
    3     & 80    & 2629.70  & 128.17  & 0.93  & 24    & 128.17  & Y \\
          & 200   & -   & -    & 9.91  & 25    & 311.80  & U \\
          & 300   & -    & -    & 28.56  & 25    & 463.50  & U \\
          & 400   & -    & -    & 74.29  & 24    & 614.64  & U \\
          & 500   & -    & -     & 141.63  & 24    & 765.41  & U \\
          \midrule
          & 10    & 1.49  & 37.35  & 0.37  & 53    & 37.35  & Y \\
          & 20    & 4.28  & 117.77  & 0.76  & 56    & 117.77  & Y \\
          & 40    & 85.18  & 358.59  & 8.07  & 59    & 358.59  & Y \\
    4     & 60    & 907.87  & 679.86  & 45.64  & 60    & 679.86  & Y \\
          & 80    & 2730.08  & 1066.19  & 123.07  & 60    & 1066.19  & Y \\
          & 90    &   -    &  -    & 131.97  & 61    & 1280.83  & Y \\
          & 100   & -    &-     & 192.51  & 61    & 1508.65  & Y \\
          & 120   &   -   &  -     & 398.78  & 61    & 2001.25  & Y \\
          \midrule
          & 5     & 0.94  & 70.76  & 0.42  & 26    & 70.76  & Y \\
    5     & 10    & 8.38  & 239.94  & 0.77  & 25    & 239.94  & Y \\
          & 20    & 545.40  & 901.64  & 16.54  & 25    & 901.64  & Y \\
          & 30    & -    & -     & 129.61  & 25    & 1964.78  & U \\
\midrule
         & 5     & 3.15  & 46.93  & 1.00  & 42    & 46.93  & Y \\
    6     & 10    & 11.13  & 404.76  & 5.96  & 56    & 404.76  & Y \\
          & 15    & 133.89  & 1174.69  & 55.38  & 57    & 1174.69  & Y \\
          & 20    & 673.92  & 2632.34  & 361.83  & 61    & 2632.34  & Y \\
				\bottomrule
			\end{tabular}%

		}{
			\caption{Structured tensors with  fraction functions \eqref{tensor:fraction}.}
			\label{tab:fraction}
		}
	\end{floatrow}
\end{table*}

\paragraph{Discussions on the results} 
Concerning the solutions, we can see that when SDP is available,  the proposed method always returns the same objective value as that of SDP. For Hilbert tensors listed in Table \ref{tab:hilbert}, when SDP is not available, namely, when SDP is too time-consuming, our method can still find the optimal solutions, which have been identified by using \eqref{eq:kkt_nonconvex_reformulation}. In fact, for Hilbert tensors, all the results meet \eqref{eq:kkt_nonconvex_reformulation_notau}. For the other three classes of tensors, when $d$ is even, our method can find the optimal solutions as well, and in fact, all the results meet \eqref{eq:kkt_nonconvex_reformulation_notau}; when $d$ is odd, it is not sure whether the results are global or not, as \eqref{eq:kkt_nonconvex_reformulation} is not met, for which we mark `U'. Nevertheless, since when $d$ is odd and when the size is not so large, our method all have found the optimal solutions, it can be indicated that when the size is large, the returned solutions may still be optimal. 

Concerning CPU time, we can see that the proposed method has a significant improvement, compared with the SDP relaxation. This is not surprising, as our method is based on only computing the leading eigenvalue/singular value of a matrix, while SDP relaxation relies on full/partial EVD, which is more time-consuming. With this advantage, we can see that the proposed method is more scalable and more efficiencient; e.g., when for Hilbert tensors of $d=3$ and $n=500$, the algorithm returns the results within 100 seconds.

Concerning the iterates, we can observe that our method always requires less than 100 iterates to reach the stopping criterion. This also helps to improve   efficiency. On the other hand, it is also interesting to see that the method is stable for the four classes of tensors, in that when  $d$ is fixed, the number of iterates does not vary a lot when $n$ increases. This helps in the scalability of the method. Of course, such a feature does not always hold, as can be seen from Table \ref{tab:randn}.

The influence of $\tau$ on a Hilbert tensor of order $4$ dimension $5$ is shown in Table \ref{tab:effect_tau2}, where $\epsilon=10^{-8}$. For all $\tau$ except $10^3$, the algorithm all finds the global solution, where for $\tau=10^3$, if setting $\epsilon=10^{-10}$, then the value is also $3.5432$. Moreover, for all   $\tau$ except $10^{3}$, the results meet the optimality condition \eqref{eq:kkt_nonconvex_reformulation_notau}.
\begin{table}[htbp]
	\centering
	\caption{Eigenvalue of a Hilbert tensor $\mathcal A\in\mathbb S^{5^4}$ with different $\tau$}
	\begin{mytabular}{ccccccccc}
		\toprule
		$\tau$ & $10^{-3}$ & $10^{-2}$ & $10^{-1}$ & $0.5$  & $1$ & $10$ & $10^2$ & $10^3$ \\
		\midrule
		$V$     & 3.5432  & 3.5432 & 3.5432  & 3.5432 & 3.5432  & 3.5432  & 3.5432   & 3.5431 \\
		Iter. & 22787  & 2286   & 236   & 50    & 23    & 74    & 126   & 28  \\
		\bottomrule
	\end{mytabular}%
	\label{tab:effect_tau2}%
\end{table}%

\subsection{Randomly generated tensors}


\begin{table}[h]\footnotesize
	\renewcommand\arraystretch{0.61}
	\centering
	\caption{Randomly nonnegative and sparse tensors with sparsity level $0.9$}
	\begin{tabular}{cc|cc|ccccc}
		\toprule
		&       & \multicolumn{2}{c}{SDP \cite{nie2014semidefinite}} & \multicolumn{5}{c}{Algorithm \ref{alg:admm}} \\
		\cmidrule{3-9}    d     & n     & Time  & V     & Time  & Iter. & V     & \#Opt & \#Meet \eqref{eq:kkt_nonconvex_reformulation_notau}\\
		\midrule
		& 20    & 9.90  & 6.82  & 0.50  & 49    & 6.82  & 50  & 50 \\
		& 40    & 151.80  & 18.20  & 1.61  & 45    & 18.20  & 50  & 50 \\
		& 60    & 512.57  & 33.10  & 2.19  & 43    & 33.10  & 50 & 50  \\
		3     & 80    & 2924.45  & 50.46  & 3.02  & 42    & 50.46  & 50  & 50 \\
		& 200   & -     &-     & 13.84  & 40    & 199.44  & 50 & 50  \\
		& 300   & -     & -    & 46.67  & 39    & 364.48  & 50  & 50 \\
		& 400   & -   & -    & 100.84  & 38    & 559.69  & 50  & 50 \\
		& 500   & -    & -     & 190.60  & 37    & 780.92  & 50  & 50 \\
		\midrule
		& 10&	2.82 &	6.78 &	2.51 &	49 &	6.78 &	50 & 50 \\    
		& 20    & 10.94  & 24.74  & 1.63  & 40    & 24.74  & 50  & 50 \\
		4     & 40    & 41.39  & 96.89  & 3.32  & 35    & 96.89  & 50  & 50 \\
		& 60    & 387.27  & 216.80  & 19.24  & 33    & 216.80  & 50 & 50  \\
		& 80    & - &-  & 59.67  & 32    & 386.98  & 50  & 50 \\
		& 100   & -    & -    & 398.75  & 31    & 597.36  & 50  & 50 \\
		\midrule
		& 5     & 2.28  & 4.26  & 0.44  & 32    & 4.26  & 50  & 19 \\
		5     & 10    & 205.88  & 18.17  & 1.49  & 21    & 18.17  & 50  & 50 \\
		& 20    & 827.38  & 97.83  & 6.33  & 20    & 97.83  & 50  & 50 \\
		& 30    &-     & -     & 41.67  & 19    & 157.01  & 50  & 50 \\
		\midrule
		& 5     & 1.48  & 7.54  & 2.92  & 67    & 7.54  & 50 & 50  \\
		6     & 10    & 21.05  & 47.81  & 8.07  & 57    & 47.81  & 50  & 50 \\
		& 15    & 149.63  & 172.69  & 64.47  & 57    & 172.69  & 50  & 50 \\
		& 20    & 408.20  & 204.13  & 262.63  & 34    & 204.13  & 50 & 50  \\
		\bottomrule
	\end{tabular}%
	\label{tab:nonnegative_sparse_tensor}%
\end{table}%

Two classes of randomly generated tensors are considered. The first one is the class of sparse and nonnegative symmetric tensors, where the entries are uniformly drawn from $[0,1]$ and symmetrized at first,   $90\%$ of which are set zero then.  It is known that a hypergraph is corresponding to a nonnegative symmetric tensor; see e.g., \cite{qi2017tensor}, and each entry represents an edge of the hypergraph. In reality, the number of edges is often very small, resulting in that the associated tensor is very sparse.    The entries of the second class of tensors are firstly drawn from the Gaussian distribution, and then the tensors are symmetrized. For each $d$ and each $n$, we run $50$ instances. 

\paragraph{Remarks on the tables}  The format is similar to the previous tables, whereas the results are presented as averages over $m$ instances for each $d$ and $n$. $m=50$ for all except   some cases that SDP relaxation cannot return   reasonable results (the optimal value is over $10^{30}$), which may be due to that the SDP solver experiences numerical troubles, as noted in \cite[p. 16]{nie2014semidefinite}. For these cases, we only average the reasonable results.
Note that in Tables \ref{tab:nonnegative_sparse_tensor} and \ref{tab:randn}, the 8th column counts the times that the proposed method finds the optimal solutions; in Table \ref{tab:nonnegative_sparse_tensor}, the last column represents the times that the proposed method identifies the optimal solutions using \eqref{eq:kkt_nonconvex_reformulation_notau}; in Table \ref{tab:randn}, the last column stands for the averaged ratio of the objective value  returned by our method to that of SDP, i.e., the ratio is given by $\sum^m_{i=1} \frac{V_{ours,i}}{V_{SDP,i}}/m$.

\begin{table}[h]\footnotesize
	\renewcommand\arraystretch{0.61}
	\centering
	\caption{Randomly tensors with   Gaussian distribution and symmetrization.}
	\begin{tabular}{cc|cc|ccccc}
		\toprule
		&       & \multicolumn{2}{c}{SDP \cite{nie2014semidefinite}} & \multicolumn{5}{c}{Algorithm \ref{alg:admm}} \\
		\cmidrule{3-9}    d     & n     & Time  & V     & Time  & Iter. & V     & \#Opt/$m$ & Ratio $\sum^m_{i=1} \frac{V_{ours,i}}{V_{SDP,i}}/m$\\
		\midrule
          & 20    & 3.60  & 6.95  & 0.49  & 99.20  & 6.85  & {39/50} & 98.50\% \\
          & 25    & 7.96  & 7.80  & 0.76  & 160.20  & 7.63  &  {31/50} & 97.76\% \\
          & 30    & 17.85  & 8.68  & 0.83  & 174.70  & 8.45  &  {25/50} & 97.38\% \\
    3     & 35    & 20.68  & 9.35  & 0.63  & 247.10  & 9.12  &  {22/50} & 97.54\% \\
          & 40    & 77.65  & 10.11  & 1.38  & 280.84  & 9.83  &  {23/50} & 97.21\% \\
          & 45    & 237.85  & 10.82  & 2.78  & 235.73  & 10.54  &  {21/50} & 97.41\% \\
          \midrule
          & 10    & 0.97  & 4.90  & 0.42  & 46.38  & 4.85  &  {40/50} & 98.94\% \\
          & 15    & 1.80  & 6.17  & 0.47  & 38.22  & 6.10  &  {38/50} & 98.91\% \\
    4     & 20    & 3.45  & 7.38  & 0.78  & 53.90  & 7.31  &  {39/50} & 99.06\% \\
          & 25    & 10.52  & 8.32  & 2.13  & 66.22  & 8.17  &  {36/50} & 98.23\% \\
          & 30    & 29.32  & 9.33  & 7.25  & 135.96  & 9.05  &  {25/47} & 97.07\% \\
          & 35    & 63.57  & 10.17  & 12.01  & 137.50  & 9.88  &  {16/32} & 97.11\% \\
          \midrule
          & 5     & 0.80  & 3.43  & 0.71  & 217.88  & 3.34  &  {39/50} & 97.04\% \\
    5     & 10    & 4.12  & 5.43  & 1.03  & 119.26  & 5.12  &  {27/50} & 94.15\% \\
          & 15    & 61.84  & 6.75  & 13.74  & 201.42  & 6.39  & 15/50 & 94.76\% \\
          & 20    & 640.68  & 7.92  & 82.96  & 344.92  & 7.17  &  {9/50} & 90.57\% \\
          \midrule
          & 5     & 0.43  & 3.45  & 0.27  & 37.04  & 3.44  & 48/50 & 99.86\% \\
    6     & 10    & 5.40  & 5.36  & 10.41  & 117.14  & 5.31  & 40/50 & 99.05\% \\
          & 15    & 100.73  & 6.93  & 153.34  & 166.85  & 6.56  & 18/49 & 94.75\% \\
		\bottomrule
	\end{tabular}%
	\label{tab:randn}%
\end{table}%

\paragraph{Discussions on the results}
From Table \ref{tab:nonnegative_sparse_tensor}, we see that for sparse and nonnegative tensors,  our method is still quite effective, which can be seen from the 8th column. On the other hand, the last column shows that for all but the case $d=5$ and $n=5$, the results returned by our method meet the optimality condition \ref{eq:kkt_nonconvex_reformulation_notau}, indicating that our method can identify the results to be optimal without relying on SDP in most cases.   The method is still efficient and scalable; except  $d=6$ and $n=5$, all other cases show that our method is faster, which is much more evident when $d=3,4$ and $n$ is large. It can also be observed that the iterates are still stable.

For the second class  of tensors, we first note that when   $(d,n)=(4,30),(4,35),(6,15)$, the SDP solver experiences numerical troubles for some instances, and the results are unreasonable. In view of this, we do not take the related instances into consideration. From Table \ref{tab:randn}, we see that not all of the global solutions of the cases can be found by our method. We can observe that when $n$ is small,     about     $80\%$ percent of the global solution of the instances can be found by our method, which can be seen from the 8th column; when $n$ increases, the percentage gradually decreases. The reason may   because that for such unstructured tensors, the gap between the largest eigenvalue/singular value and the second largest one of $\boldsymbol{ Mat }(\Lambda^*)$ is small, resulting in that it is hard to find a global solution. Nevertheless, we have observed that for all such instances, the algorithm still converges to an eigenpair, whose solution quality is still high; this can be seen from the last column, which shows that for most cases, the ratio $\sum^m_{i=1}\frac{V_{ours,i}}{V_{SDP,i}}$ is larger than $95\%$, and is close to $99\%$ when $n$ is small. Concerning   efficiency, our method is still faster except when $(d,n)=(6,10),(6,15)$, where   the efficiency is decreased by computing the symmetrization. On the iterates, unlike the other classes of tensors, now the iterates gradually increase as $n$ becomes large.

To see the influence of $\tau$, we randomly generate   $\mathcal A\in\mathbb S^{5^4}$ of the second class of tensors, where the results are illustrated in Table \ref{tab:effect_tau3}. Here $2.4775$ is the largest eigenvalue of $\mathcal A$. The results are similar to those of Table \ref{tab:effect_tau}.
\begin{table}[htbp]
	\centering
	\caption{Eigenvalue of a randomly generated $\mathcal A\in\mathbb S^{5^4}$ with different $\tau$}
	\begin{mytabular}{ccccccccc}
		\toprule
		$\tau$ & $10^{-3}$ & $10^{-2}$ & $10^{-1}$ & $0.5$  & $1$ & $10$ & $10^2$ & $10^3$ \\
		\midrule
		$V$     & 2.4775  & 2.4775 & 2.4775  & 2.4609 & 2.4609  & 2.4609  & 2.4609   & 2.4609  \\
		Iter. & 17614  & 2352   & 130   & 62    & 122    & 776    & 4209   & 19925  \\
		\bottomrule
	\end{mytabular}%
	\label{tab:effect_tau3}%
\end{table}%

\paragraph{Summary} From all the experiments especially Tables \ref{tab:hilbert}-\ref{tab:randn}, we have observed that compared with SDP relaxation, the proposed approach is  more efficient and more   scalable, due to that each iterate involves only computing the   leading eigenvalue/singular value of a certain matrix. When $n$ is large and $d\leq 5$,  the proposed method has a significant improvement considering the efficiency.
On the other hand,  the approach is effective to find the optimal solutions in most cases, especially for structured tensors. Among the cases that our method finds the optimal solutions, some of which satisfy the optimality condition \eqref{eq:kkt_nonconvex_reformulation_notau}, namely, the method automatically identifies the optimal solutions. Even if the method cannot find the global optimizer, the solution is still of high quality. In a reasonable range of $\tau$, the algorithm performs well    concerning both iterates and the solution quality, which means that it may not be very hard to choose a good $\tau$. 

In view of Theorem \ref{th:sec:eps:epsilon_optimal_solution}, it is possible to stop the algorithm earlier, e.g.,, setting $\epsilon=10^{-2}$, and then apply a more efficient local search method to   converge to the optimal solution, so as to further accelerate the method.

 \section{Concluding Remarks}\label{sec:conc}
 To tackle the problem of maximizing a homogeneous multivariate polynomial over the unit sphere, a nonconvex approach from the tensor perspective, whose goal is to be more efficient and more scalable than SDP relaxation, and to keep the effectiveness of SDP as much as possible, has been proposed in this work. The approach is built upon a nonconvex matrix program with a matrix rank-1 constraint, which is  equivalent to the original problem by revealing an equivalence property between rank-1 symmetric tensors and matrices of any order. A nonconvex ADMM is then developed to directly solve the   matrix program, whose theoretically computational complexity is linear to the input tensor. Although being nonconvex, the algorithm is proved to converge to a global optimizer under certain reasonable hypothesis. Numerical experiments on different classes of tensors demonstrate the efficiency, scalability,  and effectiveness in most cases, especially in structured tensors.
 
 Several questions remain, and several potential improvements can be made: 
 
 1) For what kinds of tensors can the hypothesis on $\Lambda^*$ of Theorem \ref{th:convergence_admm_even} hold. As discussed in Sect. \ref{sec:conv} and \ref{sec:uniqueness}, this also provides a sufficiency for the tightness of the convex relaxation.

  2) How to further improve the effectiveness on unstructured tensors. 
  
  3) The efficiency and scalability may be further improved by using more efficient algorithms for finding the leading matrix eigenvalue/singular values, e.g., using randomized algorithms. 
  
  4) We only use a vanilla ADMM, while recent advances in splitting methods may be applied. 
  
  5) While this work only focuses on symmetric tensors, it is possible to design similar approaches for finding the leading eigenvalue/singular value of partially symmetric/nonsymmetric tensors, and for copositive tensor detection, etc.
  
   6) It has been mentioned in Sect. \ref{sec:conv} that \eqref{prob:relax3} is an instance of the problem: $\min_{A\mathbf x+B\mathbf y=0} f(\mathbf x) + g(\mathbf y)$, where $f$ is nonconvex and nonsmooth, and $g$ is nonconvex and smooth, where   ${\rm Im}(B)\subset {\rm Im}(A)$ (a reverse relation in contrast to those in the literature).  How to prove the convergence of splitting methods for this kind of general problems?
  
  These will be our  further work.
  

   \bibliography{nonconvex,tensor,bib_tensor}
\bibliographystyle{plain}

  

\section*{Supplementary material}

# Supplemental Materials to: Efficiently Maximizing a Homogeneous Polynomial over Unit Sphere without Convex Relaxation

Yuning Yang[*]     Guoyin Li[†]

October 1, 2019

## Abstract

These supplemental materials provide details proofs for Theorems 2.1 and 3.7 in the paper.

## 1 Theorems and Proofs

We first need some necessary notations. The definitions of tensorization and matricization used here generalize those used in the paper.

**Tensorization and matricization**   Tensorization is defined in a unified way as follows:

1) For $\mathcal{A} \in \mathbb{R}^{n_1 \times \cdots \times n_d}$, and let $\nu_i, 1 \leq i \leq m$ be $m$ ($1 \leq m \leq d$) proper mode-sets, with $\cup_{i=1}^m \nu_i = \{1, \ldots, d\}$ and $\nu_i \cap \nu_j = \emptyset, \forall 1 \leq i, j \leq m$. Let $\text{card}(\cdot)$ denote the cardinality of a set. Denote $\tilde{\mathcal{A}} := \boldsymbol{Ten}_{[\nu_1;\cdots;\nu_m]}(\mathcal{A})$ as an $m$-th order tensor in $\mathbb{R}^{\prod_{j \in \nu_1} n_j \times \cdots \times \prod_{j \in \nu_m} n_j}$, with the indices of $\tilde{\mathcal{A}}$ corresponding to the $i$-th mode ($1 \leq i \leq m$) arranged by counting the indices corresponding to the modes in $\nu_i$ in a proper way. *Here a semicolon indicates a new mode.*

For example, for a 6-th order tensor $\mathcal{A}$, if $\nu_1 = 1, \nu_2 = 2, 3, \nu_3 = 4, \nu_4 = 5, 6$, then

$$\tilde{\mathcal{A}} = \boldsymbol{Ten}_{[\nu_1;\nu_2;\nu_3;\nu_4]}(\mathcal{A}) \in \mathbb{R}^{n_1 \times n_2 n_3 \times n_4 \times n_5 n_6}.$$

When $m = 1$, it reduces to a vectorization, i.e.,

$$\boldsymbol{Ten}_{[1,2,\ldots,d]}(\mathcal{A}) = \mathsf{reshape}(\mathcal{A}, \prod_{i=1}^d n_i, 1) \in \mathbb{R}^{\prod_{i=1}^d n_i}.$$

2) For $\tilde{\mathcal{A}} = \boldsymbol{Ten}_{[\nu_1;\cdots;\nu_m]}(\mathcal{A}) \in \mathbb{R}^{\prod_{j \in \nu_1} n_j \times \cdots \times \prod_{j \in \nu_m} n_j}$, we denote $\boldsymbol{Ten}_{[1;\cdots;d]}(\cdot)$ as the inverse of $\boldsymbol{Ten}_{[\nu_1;\cdots;\nu_m]}(\cdot)$, i.e.,

$$\mathcal{A} = \boldsymbol{Ten}_{[1;\cdots;d]}(\tilde{\mathcal{A}}) = \boldsymbol{Ten}_{[1;\cdots;d]}(\boldsymbol{Ten}_{[\nu_1;\cdots;\nu_m]}(\mathcal{A})) \in \mathbb{R}^{n_1 \times \cdots \times n_d}.$$

Throughout this work, we omit the subscript $_{[1;\cdots;d]}$ and write

$$\boldsymbol{Ten}(\cdot) = \boldsymbol{Ten}_{[1;\cdots;d]}(\cdot) \in \mathbb{R}^{n_1 \times n_2 \times \cdots \times n_d}.$$

When $m = 2$, we denote $\boldsymbol{Mat}_{[\nu_1;\nu_2]}(\cdot) := \boldsymbol{Ten}_{[\nu_1;\nu_2]}(\cdot)$ as the matricization operator. In particular, when $\nu_1 = \{1, \ldots, \lfloor d/2 \rfloor\}$ and $\nu_2 = \{1, \ldots, d\} \setminus \nu_1$, we omit the subscript and simply denote it as $\boldsymbol{Mat}(\cdot)$. This coincides with the notations in the paper.

### 1.1 Proof of Theorem 2.1 of the paper

For convenience we recall the theorem to be proved in the following.

**Theorem 1.1** (Theorem 2.1 of the paper). *For any integer $d \geq 2$, there holds*

$$\{\mathcal{X} \mid \text{rank}_{\text{CP}}(\mathcal{X}) = 1, \mathcal{X} \in \mathbb{S}^{n^d}\} = \{\mathcal{X} \mid \text{rank}(\boldsymbol{Mat}(\mathcal{X})) = 1, \mathcal{X} \in \mathbb{S}^{n^d}\}. \tag{1}$$

---

[*]College of Mathematics and Information Science, Guangxi University, Nanning, 530004, China (yuning.yang1207@gmail.com).
[†]Department of Applied Mathematics, University of New South Wales, Sydney 2052, Australia (g.li@unsw.edu.au).

*Proof.* The LHS $\subseteq$ RHS is clear. We use induction method to show the reverse side. When $d = 3$, as $\text{rank}(\boldsymbol{Mat}(\mathcal{X})) = 1$, we can write $\boldsymbol{Mat}(\mathcal{X}) := \mathbf{x}_1 \circ \mathbf{x}_{23}$ with $\mathbf{x}_1 \in \mathbb{R}^n$ and $\mathbf{x}_{23} \in \mathbb{R}^{n^2}$. By folding $\mathbf{x}_{23}$ to an $n \times n$ matrix and using SVD, $\mathbf{x}_{23}$ can be factorized as $\mathbf{x}_{23} = \sum_{i=1}^r \mathbf{x}_{2i} \otimes \mathbf{x}_{3i}$, where $[\ldots, \mathbf{x}_{ki}, \ldots]$, $k = 1, 2$ are two orthogonal matrices, and so $\mathcal{X} = \mathbf{x}_1 \circ (\sum_{i=1}^r \mathbf{x}_{2i} \circ \mathbf{x}_{3i})$. We then have the following expressions:

$$\boldsymbol{Per}_{[2;1;3]}(\mathcal{X}) = \sum_{i=1}^r \mathbf{x}_{2i} \circ \mathbf{x}_1 \circ \mathbf{x}_{3i} \text{ and } \boldsymbol{Mat}(\boldsymbol{Per}_{[2;1;3]}(\mathcal{X})) = \sum_{i=1}^r \mathbf{x}_{2i} \circ (\mathbf{x}_1 \otimes \mathbf{x}_{3i}). \quad (2)$$

Since $\boldsymbol{Mat}(\boldsymbol{Per}_{[2;1;3]}(\mathcal{X})) = \boldsymbol{Mat}(\mathcal{X})$, we get $\text{rank}(\boldsymbol{Mat}(\boldsymbol{Per}_{[2;1;3]}(\mathcal{X}))) = 1$. This together with $\mathbf{x}_{2i}$, $1 \leq i \leq r$ being orthogonal and $\mathbf{x}_1 \otimes \mathbf{x}_{3i}$, $1 \leq i \leq r$ being also orthogonal that (2) holds iff $r = 1$. Thus provided that $\text{rank}(\boldsymbol{Mat}(\mathcal{X})) = 1$, $\mathcal{X}$ is indeed a rank-1 tensor. Therefore, when $d = 3$, (1) holds. The proof for $d = 4$ is analogous. When $d = 2$, the conclusion holds naturally.

Assume that (1) holds for all $d$ with $2^m \leq d \leq 2^{m+1}$, with $m \geq 2$. We then show that it also holds for all $d$ with $2^{m+1} \leq d \leq 2^{m+2}$. Denote

$$l := \{1, \ldots, \lfloor d/2 \rfloor\} \text{ and } r := \{\lfloor d/2 \rfloor + 1, \ldots, d\}.$$

Since $d \geq 2^{m+1}$ with $m \geq 2$, we further partite $l$, $r$ as $l = ll \cup lr$ and $r = rl \cup rr$, with

$$ll = \{1, \ldots, \lfloor \text{card}(l)/2 \rfloor\}, lr = l \setminus ll;$$
$$rl = \{\lfloor d/2 \rfloor + 1, \ldots, \lfloor d/2 \rfloor + 1 + \lfloor \text{card}(r)/2 \rfloor\}, \ rr = r \setminus rl.$$

It can be seen that

$$\text{either card}(ll) = \text{card}(rl), \text{ or card}(lr) = \text{card}(rr). \quad (3)$$

For instance, when $d = 5$, $ll = \{1\}$, $lr = \{2\}$, $rl = \{3\}$, $rr = \{4, 5\}$. With such partition, $\mathcal{X}$ can be tensorized as a 4-th order tensor, namely, we set

$$\tilde{\mathcal{X}} := \boldsymbol{Ten}_{[ll;lr;rl;rr]}(\mathcal{X}) \in \mathbb{R}^{n^{\text{card}(ll)} \times n^{\text{card}(lr)} \times n^{\text{card}(rl)} \times n^{\text{card}(rr)}}.$$

On the other hand, note that $\boldsymbol{Mat}(\tilde{\mathcal{X}}) = \boldsymbol{Mat}(\mathcal{X})$, and so $\text{rank}(\boldsymbol{Mat}(\tilde{\mathcal{X}})) = 1$; then write $\boldsymbol{Mat}(\tilde{\mathcal{X}}) := \mathbf{x}_l \circ \mathbf{x}_r$. By folding $\mathbf{x}_l$ and $\mathbf{x}_r$ to $X_l \in \mathbb{R}^{n^{\text{card}(ll)} \times n^{\text{card}(lr)}}$ and $X_r \in \mathbb{R}^{n^{\text{card}(rl)} \times n^{\text{card}(rr)}}$ respectively and using again SVD, one obtaines $\mathbf{x}_l = \sum_{i=1}^{r_1} \mathbf{x}_{1i} \otimes \mathbf{x}_{2i}$ and $\mathbf{x}_r = \sum_{j=1}^{r_2} \mathbf{x}_{3j} \otimes \mathbf{x}_{4j}$. Therefore we have the expression

$$\boldsymbol{Mat}(\boldsymbol{Per}_{[3;2;1;4]}(\tilde{\mathcal{X}})) = \sum_{i=1}^{r_1} \sum_{j=1}^{r_2} (\mathbf{x}_{3j} \otimes \mathbf{x}_{2i}) \circ (\mathbf{x}_{1i} \otimes \mathbf{x}_{4j}).$$

By (3), without loss of generality assume that $\text{card}(ll) = \text{card}(rl)$; using the symmetry of $\mathcal{X}$ then, we thus have $\boldsymbol{Per}_{[3;2;1;4]}(\tilde{\mathcal{X}}) = \tilde{\mathcal{X}}$, and so $\text{rank}(\boldsymbol{Mat}(\boldsymbol{Per}_{[3;2;1;4]}(\tilde{\mathcal{X}}))) = 1$. Similar to the $d = 3$ case, we deduce that $r_1 = r_2 = 1$. As a result, $\text{rank}((X_l)) = 1$ and $\text{rank}((X_r)) = 1$.

If writing $\mathcal{X}_l := \boldsymbol{Ten}(\mathbf{x}_l) \in \mathbb{S}^{n^{\text{card}(l)}}$ and $\mathcal{X}_r := \boldsymbol{Ten}(\mathbf{x}_r) \in \mathbb{S}^{n^{\text{card}(r)}}$ where $2^m \leq \text{card}(l), \text{card}(r) \leq 2^{m+1}$, then it follows that $\text{rank}(\boldsymbol{Mat}(\mathcal{X}_l)) = 1$ and $\text{rank}(\boldsymbol{Mat}(\mathcal{X}_r)) = 1$. This together with the symmetry of $\mathcal{X}_l$ and $\mathcal{X}_r$ and the assumption implies that $\text{rank}_{\text{CP}}(\mathcal{X}_l) = \text{rank}_{\text{CP}}(\mathcal{X}_r) = 1$. As $\boldsymbol{Mat}(\tilde{\mathcal{X}}) = \mathbf{x}_l \circ \mathbf{x}_r$, it follows that $\tilde{\mathcal{X}}$ itself is also a rank-1 tensor. Therefore, induction method shows that RHS $\subseteq$ LHS, and so (1) holds for all $d \geq 2$. This completes the proof. $\square$

## 1.2 Proof of Theorem 3.7 of the paper

For convenience we recall the theorem to be proved and the assumptions in the following.

**Assumption 1.1.** *1. $\{\mathcal{X}^k, \mathcal{Y}^k, \Lambda^k\}$ is feasible to the problem under consideration;*

*2. $\|\{\mathcal{X}^k, \mathcal{Y}^k, \Lambda^k\} - \{\mathcal{X}^{k-1}, \mathcal{Y}^{k-1}, \Lambda^{k-1}\}\|_F = \epsilon$, with $\epsilon > 0$ sufficiently small.*

**Assumption 1.2.** *$\mathcal{X}^k$ is partially symmetric.*

**Theorem 1.2** (Theorem 3.7 of the paper)**.** *Let $d$ be fixed. Let $\{\mathcal{X}^k, \mathcal{Y}^k, \Lambda^k\}$ satisfy Assumptions 1.1 and 1.2. Then in polynomial time, one can find a symmetric normalized rank-1 tensor $\overline{\mathcal{X}}$, namely, $\overline{\mathcal{X}}$ is feasible to (2.4), such that $\|\mathcal{X}^k - \overline{\mathcal{X}}\|_F = O(\epsilon)$.*

To prove Theorem 1.2, we need some lemmas.

**Lemma 1.1.** *Let $\mathcal{X}, \mathcal{Y}$ be 4-th order tensors, with $\mathcal{Y}$ satisfying $\mathcal{Y} = \boldsymbol{Per}_{[3;2;1;4]}(\mathcal{Y})$ or $\mathcal{Y} = \boldsymbol{Per}_{[1;4;3;2]}(\mathcal{Y})$; $\mathrm{rank}(\boldsymbol{Mat}(\mathcal{X})) = 1$ and $\|\mathcal{X}\|_F = 1$. Assume that $\mathcal{Y} - \mathcal{X} = O(\epsilon)$ with $\epsilon$ being sufficiently small. Then one can find a normalized 4-th order rank-1 tensor $\mathcal{X}'$ in polynomial time, such that $\mathcal{X} - \mathcal{X}' = O(\epsilon)$.*

*Proof.* If $\mathrm{rank}_{\mathrm{CP}}(\mathcal{X}) = 1$, then setting $\mathcal{X}' = \mathcal{X}$ we obtain the desired results. Thus in what follows, we suppose $\mathrm{rank}_{\mathrm{CP}}(\mathcal{X}) > 1$. By the assumption, we first write $\boldsymbol{Mat}(\mathcal{X}) = \mathbf{x}_l \mathbf{x}_r^\top$ with $\|\mathbf{x}_l\| = \|\mathbf{x}_r\| = 1$. Using SVD, we have

$$\mathcal{X} = \left(\sum_{i=1} \lambda_i \mathbf{x}_{1i} \circ \mathbf{x}_{2i}\right) \circ \left(\sum_{j=1} \sigma_j \mathbf{x}_{3j} \circ \mathbf{x}_{4j}\right) = \sum_{i,j=1} \lambda_i \sigma_j \mathbf{x}_{1i} \circ \mathbf{x}_{2i} \circ \mathbf{x}_{3j} \circ \mathbf{x}_{4j}, \quad (4)$$

where $[\ldots, \mathbf{x}_{ki}, \ldots]$, $k = 1, \ldots, 4$ are orthonormal matrices; $\lambda_i \geq 0$ and $\sigma_j \geq 0$ denote the singular values, arranged in a descending order. Denote $\mathcal{X}' := \mathbf{x}_{11} \circ \mathbf{x}_{21} \circ \mathbf{x}_{31} \circ \mathbf{x}_{41}$. We show that $\mathcal{X}'$ is the desired rank-1 tensor. It follows from $\|\mathbf{x}_l\| = \|\mathbf{x}_r\| = 1$ that

$$\sum_{i=1} \lambda_i^2 = 1, \quad \sum_{j=1} \sigma_j^2 = 1. \quad (5)$$

By the assumption, without loss of generality we assume that $\mathcal{Y} = \boldsymbol{Per}_{[3;2;1;4]}(\mathcal{Y})$. Then it holds that

$$\boldsymbol{Per}_{[3;2;1;4]}(\mathcal{X}) - \mathcal{X} = \boldsymbol{Per}_{[3;2;1;4]}(\mathcal{Y} - O(\epsilon)) - (\mathcal{Y} - O(\epsilon)) = O(\epsilon),$$

and so

$$O(\epsilon^2) = \frac{1}{2}\|\boldsymbol{Per}_{[3;2;1;4]}(\mathcal{X}) - \mathcal{X}\|_F^2 = 1 - \langle \boldsymbol{Per}_{[3;2;1;4]}(\mathcal{X}), \mathcal{X}\rangle, \quad (6)$$

where the second equation comes from that $\|\mathcal{X}\|_F = 1$. On the other hand, using (4), $\boldsymbol{Per}_{[3;2;1;4]}(\mathcal{X})$ can be represented as

$$\boldsymbol{Per}_{[3;2;1;4]}(\mathcal{X}) = \sum_{k,l=1} \lambda_k \sigma_l \mathbf{x}_{3l} \circ \mathbf{x}_{2k} \circ \mathbf{x}_{1k} \circ \mathbf{x}_{4l}. \quad (7)$$

It follows from (5), (4) and (7) that

$$\begin{aligned}
1 - O(\epsilon^2) &= \langle \mathcal{X}, \boldsymbol{Per}_{[3;2;1;4]}(\mathcal{X})\rangle \\
&= \left\langle \sum_{i,j=1} \lambda_i \sigma_j \mathbf{x}_{1i} \circ \mathbf{x}_{2i} \circ \mathbf{x}_{3j} \circ \mathbf{x}_{4j}, \sum_{k,l=1} \lambda_k \sigma_l \mathbf{x}_{3l} \circ \mathbf{x}_{2k} \circ \mathbf{x}_{1k} \circ \mathbf{x}_{4l} \right\rangle \\
&= \sum_{i,j,k,l=1} \lambda_i \sigma_j \lambda_k \sigma_l \langle \mathbf{x}_{1i}, \mathbf{x}_{3l}\rangle \langle \mathbf{x}_{2i}, \mathbf{x}_{2k}\rangle \langle \mathbf{x}_{3j}, \mathbf{x}_{1k}\rangle \langle \mathbf{x}_{4j}, \mathbf{x}_{4l}\rangle \\
&= \sum_{i,j=1} \lambda_i^2 \sigma_j^2 \langle \mathbf{x}_{1i}, \mathbf{x}_{3j}\rangle^2 =: \sum_{i,j=1} \lambda_i^2 \sigma_j^2 \delta_{ij}^2,
\end{aligned} \quad (8)$$

where the fourth equation is due to that $\mathbf{x}_{ki}$ are orthonormal, and we denote $\delta_{ij} = \langle \mathbf{x}_{1i}, \mathbf{x}_{3j}\rangle$. Then it holds that

$$\sum_{i=1} \delta_{ij}^2 = 1, \text{ and } \sum_{j=1} \delta_{ij}^2 = 1. \quad (9)$$

It follows from (5) that (8) can be written as

$$\sum_{i=1} \lambda_i^2 \left(1 - \sum_{j=1} \sigma_j^2 \delta_{ij}^2\right) = 1 - \sum_{i,j=1} \lambda_i^2 \sigma_j^2 \delta_{ij}^2 = O(\epsilon^2). \quad (10)$$

According to (5), there exists at least an index such that $\lambda_i > c > 0$ with $c$ being a constant. Without loss of generality assume that $\lambda_1 > c$. On the other hand we have $1 - \sum_{j=1} \sigma_j^2 \delta_{ij}^2 \geq 0$ uder the constraints (5) and (9). Therefore, (10) holds if and only if

$$\sum_{j=1} \sigma_j^2 \delta_{1j}^2 = 1 - O(\epsilon^2), \quad (11)$$

while (11) holds if and only if the leading singular value $\sigma_1^2 = 1 - O(\epsilon^2)$ and $\delta_{11}^2 = 1 - O(\epsilon^2)$. Similarly, (8) can also be written as

$$\sum_{j=1} \sigma_j^2 \left(1 - \sum_{i=1} \lambda_i^2 \delta_{ij}^2\right) = O(\epsilon^2),$$

from which we also deduce that $\lambda_1^2 = 1 - O(\epsilon^2)$. It then follows that $\lambda_1 = 1 - O(\epsilon^2)$ and $\sigma_1 = 1 - O(\epsilon^2)$, and so

$$\langle \mathbf{x}_l, \mathbf{x}_{11} \otimes \mathbf{x}_{21}\rangle = \lambda_1 = 1 - O(\epsilon^2), \text{ and } \langle \mathbf{x}_r, \mathbf{x}_{31} \otimes \mathbf{x}_{41}\rangle = \sigma_1 = 1 - O(\epsilon^2). \quad (12)$$

Using (12), we thus obtain

$$\begin{aligned}
\frac{1}{2}\|\mathcal{X} - \mathcal{X}'\|_F^2 &= 1 - \langle \mathcal{X}, \mathcal{X}'\rangle \\
&= 1 - \langle \mathbf{x}_l, \mathbf{x}_{11} \otimes \mathbf{x}_{21}\rangle \langle \mathbf{x}_r, \mathbf{x}_{31} \otimes \mathbf{x}_{41}\rangle \\
&= 1 - \lambda_1 \sigma_1 = 1 - (1 - O(\epsilon^2)) = O(\epsilon^2).
\end{aligned}$$

Therefore, $\mathcal{X}'$ is the desired rank-1 tensor, and the proof is completed. □

**Lemma 1.2.** *Given a fixed order $d \geq 2$, let $\mathcal{X}, \mathcal{Y} \in \mathbb{R}^{n^d}$, with $\mathcal{Y} \in \mathbb{S}^{n^d}$ and $\mathbf{Mat}(\mathcal{X}) = \mathbf{x}_l \mathbf{x}_r^\top$, with $\mathbf{x}_l, \mathbf{x}_r$ being normalized and $\mathbf{Ten}(\mathbf{x}_l) \in \mathbb{S}^{n^{\lfloor d/2 \rfloor}}$, $\mathbf{Ten}(\mathbf{x}_r) \in \mathbb{S}^{n^{\lceil d/2 \rceil}}$. If $\mathcal{Y} - \mathcal{X} = O(\epsilon)$, then one can extract a normalized rank-1 tensor $\mathcal{X}'$ in polynomial time, such that $\|\mathcal{X}'\|_F = 1$ and $\mathcal{X} - \mathcal{X}' = O(\epsilon)$.*

*Proof.* We use induction method to prove the conclusion. First we claim that when $d \leq 4$, the results hold. When $d = 4$, Lemma 1.1 already yields the results in question. For $d = 3$ tensors $\mathcal{X}, \mathcal{Y}$, it follows that they can be equivalently regarded as 4-th order tensors $\tilde{\mathcal{X}}, \tilde{\mathcal{Y}}$, by treating the dimension of the first mode of $\tilde{\mathcal{X}}, \tilde{\mathcal{Y}}$ as 1. With such order-lift, it can be verified that $\tilde{\mathcal{X}}$ and $\tilde{\mathcal{Y}}$ still satisfy the conditions of Lemma 1.1, and the results follow. When $d = 2$, $\mathcal{X}$ itself is the required rank-1 tensor.

Assume that the conclusion holds for all $d$ with $2^m \leq d \leq 2^{m+1}$, with $m \geq 2$. We then show that it also holds for all $d$ with $2^{m+1} \leq d \leq 2^{m+2}$. The following settings are the same as Theorem 1.1. Denote

$$l := \{1, \ldots, \lfloor d/2 \rfloor\} \text{ and } r := \{\lfloor d/2 \rfloor + 1, \ldots, d\}.$$

Set $l = ll \cup lr$ and $r = rl \cup rr$ with $ll = \{1, \ldots, \lfloor \text{card}(l)/2 \rfloor\}$, $lr = l \setminus ll$; $rl = \{\lfloor d/2 \rfloor + 1, \ldots, \lfloor d/2 \rfloor + 1 + \lfloor \text{card}(r)2 \rfloor\}$, and $rr = r \setminus rl$. Denote

$$\tilde{\mathcal{X}} := \mathbf{Ten}_{[ll;lr;rl;rr]}(\mathcal{X}), \text{ and } \tilde{\mathcal{Y}} = \mathbf{Ten}_{[ll;lr;rl;rr]}(\mathcal{Y}).$$

Then, since $\mathcal{Y}$ is symmetric and either $\text{card}(ll) = \text{card}(rl)$ or $\text{card}(lr) = \text{card}(rr)$, it follows that

$$\text{either } \mathbf{Per}_{[3;2;1;4]}(\tilde{\mathcal{Y}}) = \tilde{\mathcal{Y}}, \text{ or } \mathbf{Per}_{[1;4;3;2]}(\tilde{\mathcal{Y}}) = \tilde{\mathcal{Y}}; \tag{13}$$

on the other hand, it can be seen that $\tilde{\mathcal{X}}$ meets the conditions of Lemma 1.1. Write $\mathbf{Mat}(\tilde{\mathcal{X}}) = \mathbf{x}_l \mathbf{x}_r^\top$; applying Lemma 1.1 to $\tilde{\mathcal{X}}, \tilde{\mathcal{Y}}$, we obtain a 4-th order rank-1 tensor $\overline{\mathcal{X}} = \mathbf{x}_{l1} \circ \mathbf{x}_{l2} \circ \mathbf{x}_{r1} \circ \mathbf{x}_{r2}$, where $\mathbf{x}_{l1}, \mathbf{x}_{l2}$ and $\mathbf{x}_{r1}, \mathbf{x}_{r2}$ respectively correspond to the leading singular vector pair of $\mathbf{x}_l$ and $\mathbf{x}_r$ (in matrices form), such that

$$\|\mathcal{X} - \mathbf{Ten}_{[1;2;\cdots;d]}(\overline{\mathcal{X}})\|_F = \|\tilde{\mathcal{X}} - \overline{\mathcal{X}}\|_F = O(\epsilon). \tag{14}$$

We then show that the two lower order tensors $\mathbf{Ten}(\mathbf{x}_{l1} \mathbf{x}_{l2}^\top)$ and $\mathbf{Ten}(\mathbf{x}_l)$ still meet the conditions of the lemma to be proved, namely, they respectively play the role of $\mathcal{X}$ and $\mathcal{Y}$. First, the proof of Lemma 1.1 (see (12)) also tells us that

$$\|\mathbf{x}_l - \mathbf{x}_{l1} \otimes \mathbf{x}_{l2}\| = O(\epsilon).$$

Then, since $\mathbf{Ten}(\mathbf{x}_l) \in \mathbb{S}^{n^{\lfloor d/2 \rfloor}}$ and $\mathbf{x}_{l1}, \mathbf{x}_{l2}$ is the leading singular vector pair, we get that

$$\mathbf{Ten}(\mathbf{x}_{l1}) \in \mathbb{S}^{n^{\text{card}(ll)}} \text{ and } \mathbf{Ten}(\mathbf{x}_{l2}) \in \mathbb{S}^{n^{\text{card}(lr)}}.$$

As a consequence, $\mathbf{Ten}(\mathbf{x}_{l1} \mathbf{x}_{l2}^\top)$ and $\mathbf{Ten}(\mathbf{x}_l)$ meet the conditions of this lemma. Since $2^m \leq \lfloor d/2 \rfloor \leq 2^{m+1}$, induction tells us that there exists a normalized rank-1 tensor $\mathcal{X}'_l = \mathbf{x}_1 \circ \cdots \circ \mathbf{x}_{\lfloor d/2 \rfloor}$ of order $\lfloor d/2 \rfloor$, such that

$$\|\mathbf{Ten}(\mathbf{x}_{l1} \mathbf{x}_{l2}^\top) - \mathcal{X}'_l\|_F = O(\epsilon) \Leftrightarrow \langle \mathbf{Ten}(\mathbf{x}_{l1} \mathbf{x}_{l2}^\top), \mathcal{X}'_l \rangle = 1 - O(\epsilon^2). \tag{15}$$

Analogously, one can prove that $\mathbf{Ten}(\mathbf{x}_{r1} \mathbf{x}_{r2}^\top)$ and $\mathbf{Ten}(\mathbf{x}_r)$ also meet the required conditions, and so there also exists a normalized rank-1 tensor $\mathcal{X}'_r = \mathbf{x}_{\lfloor d/2 \rfloor + 1} \circ \cdots \circ \mathbf{x}_d$ of order $\lceil d/2 \rceil$, such that

$$\|\mathbf{Ten}(\mathbf{x}_{r1} \mathbf{x}_{r2}^\top) - \mathcal{X}'_r\|_F = O(\epsilon) \Leftrightarrow \langle \mathbf{Ten}(\mathbf{x}_{r1} \mathbf{x}_{r2}^\top), \mathcal{X}'_r \rangle = 1 - O(\epsilon^2). \tag{16}$$

Combining (15) and (16) we obtain

$$\begin{aligned}
\frac{1}{2}\|\mathbf{Ten}_{[1;\cdots;d]}(\overline{\mathcal{X}}) - \mathcal{X}'_l \circ \mathcal{X}'_r\|_F^2 &= 1 - \langle \mathbf{Ten}_{[1;\cdots;d]}(\mathbf{x}_{l1} \circ \mathbf{x}_{l2} \circ \mathbf{x}_{r1} \circ \mathbf{x}_{r2}), \mathcal{X}'_l \circ \mathcal{X}'_r \rangle \\
&= 1 - \langle \mathbf{Ten}(\mathbf{x}_{l1} \mathbf{x}_{l2}^\top), \mathcal{X}'_l \rangle \langle \mathbf{Ten}(\mathbf{x}_{r1} \mathbf{x}_{r2}^\top), \mathcal{X}'_r \rangle \\
&= O(\epsilon^2).
\end{aligned} \tag{17}$$

Denote

$$\mathcal{X}' := \mathcal{X}'_l \circ \mathcal{X}'_r = \mathbf{x}_1 \circ \cdots \mathbf{x}_d.$$

Combining (14) and (17) we arrive at that for all $d \in [2^{m+1}, 2^{m+2}]$, there holds

$$\begin{aligned}
\|\mathcal{X} - \mathcal{X}'\|_F &= \|\mathcal{X} - \mathcal{X}'_l \circ \mathcal{X}'_r\|_F \\
&\leq \|\mathcal{X} - \mathbf{Ten}_{[1;\cdots;d]}(\overline{\mathcal{X}})\|_F + \|\mathbf{Ten}_{[1;\cdots;d]}(\overline{\mathcal{X}}) - \mathcal{X}'_l \circ \mathcal{X}'_r\|_F = O(\epsilon).
\end{aligned}$$

Finally, as $d$ is fixed, the value behind the $O(\cdot)$ is a constant. Thus, induction method shows that the assertion holds for all fixed $d$. Since all the execution involves SVD only, $\mathcal{X}'$ can be computed in polynomial time. This completes the proof. □

*Proof of Theorem 1.2.* It follows from Assumption 1.1 and $\Lambda^k - \Lambda^{k-1} = \tau(\mathcal{Y}^k - \mathcal{X}^k)$ that $\mathcal{X}^k - \mathcal{Y}^k = O(\epsilon)$. On the other hand, since $\mathcal{X}^k \in C \cap \mathbb{S}_P^{n^d}$ and $\mathcal{Y}^k \in \mathbb{S}^{n^d}$, Lemma 1.2 holds. Let $\mathcal{X}' = \mathbf{x}_1 \circ \cdots \circ \mathbf{x}_d$ be as that in Lemma 1.2. Denote $\overline{\mathcal{X}} := \mathbf{x}_1^{\circ^d}$. We show that $\overline{\mathcal{X}}$ is a desired rank-1 tensors. From Lemma 1.2, we know that $\|\mathcal{X}' - \mathcal{Y}^k\|_F = \|\mathcal{X}' - \mathcal{X}^k\|_F = O(\epsilon)$. Using the symmetry of $\mathcal{Y}^k$, we get for $2 \leq i \leq d$, when exchanging the orders $i$ and $1$,
$$\|\boldsymbol{Per}_{[i;2;\cdots;i-1;1;i+1;\cdots;d]}(\mathcal{X}') - \mathcal{X}'\|_F = O(\epsilon),$$
which implies that
$$\langle \mathbf{x}_1, \mathbf{x}_i \rangle^2 = \langle \boldsymbol{Per}_{[i;2;\cdots;i-1;1;i+1;\cdots;d]}(\mathcal{X}'), \mathcal{X}' \rangle = 1 - O(\epsilon^2), \ i = 2, \ldots, d.$$
Therefore,
$$\frac{1}{2}\|\overline{\mathcal{X}} - \mathcal{X}'\|_F^2 = 1 - \langle \overline{\mathcal{X}}, \mathcal{X}' \rangle = 1 - \prod_{i=2}^d \langle \mathbf{x}_1, \mathbf{x}_i \rangle = O(\epsilon^2).$$
Thus $\|\mathcal{X}^k - \overline{\mathcal{X}}\|_F = \|\mathcal{X}^k - \mathcal{X}' + \mathcal{X}' - \overline{\mathcal{X}}\|_F = O(\epsilon)$. This completes the proof. □



\end{document}